\documentclass{amsart}
\usepackage{times, xcolor}
\usepackage{amsmath,amsthm, amssymb,bbm}
\usepackage{mathrsfs}

\newcommand{\lp}{L^p(\mathbb{R}^n)}
\newcommand{\R}{\mathbb{R}}
\newcommand{\C}{\mathbb{C}}

\newcommand{\N}{\mathbb{N}}

\newcommand{\M}{{\mathcal{M}_{n,k}}}
\newcommand{\G}{{\mathcal{G}_{n,k}}}
\newcommand{\lqm}{L^q(\mathcal{M}_{n,k})}
\newcommand{\Tel}{\mathcal{T}}
\newcommand{\Tnk}{T_{n,k}}
\newcommand{\Tt}{T^*_{n,k}}
\newcommand{\Sel}{\mathcal{S}}

\newcommand{\one}{\mathbbm{1}}

\newcommand{\derivsy}[1]{\mathscr{D}^{#1}} 
\newcommand{\derivsly}[2]{\mathscr{D}_{#2}^{#1}}
\newcommand{\derivsx}[1]{D^{#1}}
\newcommand{\derivslx}[2]{D_{#2}^{#1}}

\newtheorem{thm}{Theorem}[section]
\newtheorem{rmk}[thm]{Remark}
\newtheorem{defn}[thm]{Definition}

\newtheorem{prop}[thm]{Proposition}
\newtheorem{lem}[thm]{Lemma}

\begin{document}

\title {Smoothness of extremizers for certain inequalities\\
of the Radon transform}
\author{ Taryn C. Flock }
\address{Macalester College\\Mathematics, Statistics, and Computer Science\\
Olin-Rice Science Center, Room 222\\
Saint Paul, MN 55105-1899\\USA}
\email{tflock@macalester.edu}

\begin{abstract}
 The Radon transform is a bounded operator from $L^p$ of Euclidean space to $L^q$ of the manifold of all affine hyperplanes  in $\R^n$ for certain exponents depending dimension.  Extremizers, functions which maximize the functional 
 \[ \|\mathcal{R}f\|_{\lqm}/\|f\|_{L^p(\R^n)},\] have been determined for certain values of $q$ and $p$, see \cite{BL97, C11}, but most remain open. 	

 We show that extremizers are infinitely differentiable whenever the exponents in the associated Euler-Lagrange equation, $q-1$ and $\frac1{p-1}$, are integers. The proof adapts the method of Christ and Xue \cite{CX}, to the case where the underlying space is a manifold. 

 The proof is carried out in the setting of the $k$-plane transform, which takes functions on $\R^n$ to functions on the manifold of all affine $k$-planes  in $\R^n$ by integrating the function over the $k$-dimensional plane.  We show that when $q-1$ and $\frac1{p-1}$ are intergers, all nonnegative critical points of the functional 
 \[ \|\Tnk f\|_{\lqm}/\|f\|_{L^p(\R^n)}\] 
 are infinitely differentiable, all derivatives are in $L^p$ and exhibit some additional decay measured in a weighted $L^p$-space.

\end{abstract}

\maketitle

\section{Introduction}

  Optimal constants and extremizers, the functions which achieve them, have been determined for many key inequalities in harmonic analysis, for example in Lieb's celebrated work on the Hardy-Littlewood-Sobolev inequality \cite{HLS}.  In this paper we consider extremizers for $L^p$-$L^q$ inequalities of the Radon transform, and its generalization, the $k$-plane transform. 

The Radon transform takes functions on $\R^n$ to functions on the manifold of all affine hyperplanes in $\R^n$ by mapping a function $f$ to the function whose value on a given $n-1$-plane is the integral of $f$ over that hyperplane with respect to (the restriction of) Lebesgue measure. The $k$-plane transform generalizes this operator by replacing $n-1$-dimensional planes with $k$-planes where $k\leq n$. When $k=1$, $f$ is mapped to the function whose value on a given line is the integral of $f$ over that line. This transformation is called the X-ray transform. 

Through a series of works \cite{C84,D84,OS82,Cald83,S81} the $k$-plane transform is now known to be a bounded operator from $\lp$ to $L^q$ of the manifold of all affine $k$-planes in $\R^n$ when $q\in[1,n+1]$ and $p=\frac{nq}{n-k+kq}.$ .   

In several special cases the extremizers of these $L^p$-$L^q$ inequalities have been explicitly identified. In the cases $q=2$ and $q=n+1$ all extremizers are known \cite{BL97, C11, D11, F13}. When $k=2$ all radial extremizers are known \cite{BL97}. In each of these cases, the extremizers are uniquely given by $ f(x)=(1+|x|^2)^{\frac{-(n-k)}{2(p-1)}}$, up to the symmetries of the problem. Notably, these extremizers are all infinitely differentiable. The situation is considerably different in the case that $q=1.$ Here, every nonnegative function is an extremizer regardless of differentiablity, and all extremizers are given by $g=cf(x)$ where $c\in\C$ and $f$ is nonnegative.

As a first step toward characterizing extremizers of the endpoint $q=n+1$ inequality for the Radon transform in \cite{C11}, Christ shows that all extremizers of the endpoint Radon transform inequality are smooth.  This is done by showing an identity relating extremizers for the endpoint inequality of the Radon transform, to extremizers of an $L^p$-$L^q$ estimate for a convolution operator shown to be smooth in \cite{CX}.  However, such an identity does not hold when either $q$ or $k$ is varied.

This paper adapts the methods of Christ and Xue \cite{CX} to apply directly to inequalities for the $k$-plane transform. Because the method fundamentally uses multilinearity, our results address $L^p$-$L^q$ inequalities where $\frac{1}{p-1}$ and $q-1$ are both positive integers.

	\begin{thm}\label{thm:justextremizers}  
		Let $n\geq 2$, $q\in(1,n+1]$ and $p=\frac{nq}{n-k+kq}$, such that both $q-1$ and $\frac{1}{p-1}$ are { integers}. If $f\in L^{p}(\R^n)$ is an extremizer of \eqref{LpLq}, then $f\in C^\infty$, all partial derivatives of $f$ are bounded, and there exists $\delta>0$ such that $(1+|x|^2)^\delta \derivsx{s}f(x)\in L^{p}(\R^n)$ for all $s\geq 0$.  
	\end{thm} 

We also prove that extremizers do indeed exist whenever $q$ is a positive integer, using a concentration compactness argument similar to those in \cite{HLS,D11,D12}.   Because the proof is relatively routine we defer this result to an appendix: Theorem  \ref{thm:integerexistence}. 

The fundamental approach from Christ and Xue \cite{CX} is to use that all extremizers are solutions to a certain Euler-Lagrange equation. They then show that all solutions to the Euler-Lagrange equation enjoy better decay than that of $L^p$, and all solutions enjoying better decay are smooth.  Our proof follows the same outline. One change is that while the convolution operator considered by Christ and Xue maps functions on $R^n$ to functions on $R^n$, for the $k$-plane transform one of these underlying spaces is a manifold. Both halves of the argument require additional adaptation. 

To show that solutions enjoy better decay requires weighted estimates for the $k$-plane transform. Such inequalities have been studied previously by Solmon \cite{S79} and Rubin \cite{R12}.  Weighted inequalities corresponding to known mixed norm estimates for the $k$-plane transform of radial functions have been studied by Kumar and Ray \cite{KR12}.  Solmon \cite{S79} considers weights $v_{\alpha}(x)= \langle x \rangle^{\alpha -d}$ where $0<\alpha\leq k<d$.  He shows that for $1\leq p<n/k$ and $0<\alpha<k$, $\Tnk$ is a bounded operator from $L^p$ to $L^1(v_\alpha)$.  Rubin \cite{R12} considers weights of the form $|x|^{\mu}$ where $\mu>k-n/p$. He shows that for $1\leq p \leq \infty$, $1/p+1/p'=1$, $\nu=\mu-k/p'$ and   $\mu>k-n/p$,  $\Tnk$ is a bounded operator from $L^p(|x|^\mu)$ to $L^p(|y|^\nu)$. We consider weights $\langle x\rangle^{\frac{d-k}{p-1}}$. Note that unlike in Solmon's case the power is positive. 
 \begin{prop}\label{prop:weighted}
 $\Tnk$ is a bounded operator from  $L^p(\langle x\rangle^{\frac{n-k}{p-1}})$ to $L^q(\langle y\rangle^{d})$ where $q\in(1,n+1]$ and $p=\frac{nq}{n-k+kq}$. 
 \end{prop}

The proof that solutions enjoying better decay are smooth has two key ingredients: a known smoothing result for the $k$-plane transform \cite{S81} adapted to the weighted spaces, and uniform estimates for mollified derivatives in these weighted spaces.   We correct a small technical error from \cite{CX} in the definition of these mollified derivatives and re-prove their Kato-Ponce type inequality for these mollified derivatives on weighted space (Lemma \ref{katoponce}) and prove a version where $R^n$ is replaced by the manifold of affine $k$-planes (Lemma \ref{mdgood}). For standard fractional derivatives, a similar weighted Lebesgue space result was proved in \cite{CUN16},\cite{CUN22err} for a larger range of exponents.  Our proof follows the methods of Christ and Weinstein in \cite{CW91} adapted to the weighted setting. As we are specifically interested in the case where the weights are of the form $\langle x \rangle^\alpha$, it might be possible to prove the result for a wider range of weights and exponents as in \cite{OhWu22} but we do not pursue this here.

While the condition that both $q-1$ and $\frac{1}{p-1}$ are  integers is somewhat restrictive, it is satisfied infinitely often. For example,
for the the Radon transform both are integers precisely when $q$ is an integer larger than 1 and $ n$ is divisible by $q-1$, e.g when $q=3$ and the dimension is even. However for the X-ray transform the only new case for which this condition  holds is $q=3$ and $n=4$.  
More generally, for any pair of integers $q_{el}$, $p_{el}$ not both $1$, for any $s\in\N$ taking $q=q_{el}+1$, $n=q_{el}(p_{el}+1)s$, and $k=(p_{el}q_{el}-1)s$ gives $\frac{1}{p-1}=p_{el}$ and $q-1=q_{el}$, although other choices of $n$, $q$ and $k$ may also yield this $q_{el},p_{el}$ pair.

Smoothness of extremizers for convolution with compactly supported measures has recently been shown in \cite{T23}. While the operators are related to those considered in \cite{CX} the methods are different. 
Smoothness of extremizers has also been investigated in the context of sharp Fourier restriction \cite{DQ}. Both of these example feature underlying structure the convolution st

\subsection*{Definitions and Notations} 
Before formally stating our results, we state some relatively standard definitions and notation that will be used throughout. 

Let $\G$ be the Grassmann manifold of all $k$-dimensional subspaces of $\R^n$ (namely those $k$-planes which pass through the origin)  and let $\M$ be the manifold of all affine $k$-planes in $\R^n$. We identify each $k$-plane $\M$ as a translation of a $k$-plane through the origin, which yields the parameterization of $\M$ as $(\theta,y)$ where $\theta\in\mathcal{G}_{n,k}$ and $y\in\theta^\perp$, the $(n-k)$-dimensional subspace orthogonal to $\theta$. 

The $k$-plane transform takes complex-valued functions on $\R^n$ to complex-valued functions on this manifold $\M$,  by mapping a function $f$ to the the function whose value on a given $k$-plane is the integral of $f$ over that $k$-plane with respect to (the restriction of) Lebesgue measure. 

More precisely, 
	\[ \Tnk f(\theta,y)=\int_{x\in\theta}f(x+y)\;d\lambda_\theta(x), \] 

where $d\lambda_\theta$ denotes Lebesgue measure on the $k$-plane $\theta$.   

In addition to the usual $L^p(\R^n)$ spaces we'll use Lebesgue spaces on the manifold $\M$. The parameterization of $\M$ as $\G\times\R^{n-k}$, suggests a natural product measure.  The compact manifold $\mathcal{G}_{n,k}$ is invariant under the action of the orthogonal group.  This invariance gives rise to a natural probability measure on $\G$, which we denote by $d\theta$. We let $d\lambda_{\theta^\perp}$ denote Lebesgue measure on the $(n-k)$-plane $\theta^\perp$ and equip $\M$ with the product measure $d\lambda_{\theta^\perp} d\theta$. Thus the Lebesgue spaces are the manifold $\M$ have $L^q$ norm: 
\[ \| g \|_{\lqm}= 	 \left( \int_{\G}\int_{\theta^\perp}|g(\theta,y)|^{q}\;d\lambda_{\theta^\perp}(y)\;d\theta \right)^{1/q}.\]

We will also work with the dual or adjoint $k$-plane transform given by 
	\[ \Tnk^*f(x)=\int_{\G}f(\theta,P(x,\theta^\perp))\;d\theta \] 
 where $P(x,\theta^\perp)$ is the projection of $x$ on to the $(n-k)$-dimensional plane perpendicular to $\theta$.   This is the dual transform in the sense that if $f$ and $g$ are Schwartz functions, then 
 \[ \int_{\M} (\Tnk f )\cdot\bar{g}\;d\lambda_\theta(y)\;d\theta= \int_{\R^n}f\cdot\overline{(T^*_{n,k}g)}\;dx. \]

 As proved in \cite{C84} building on the work of Drury in \cite{D83}, for $q\in[1,n+1]$ and $p=\frac{nq}{n-k+kq}$ there exists a constant $C$ such that for all  $f\in L^p(\R^n)$ 
	\begin{equation}\label{LpLq}
	\|\Tnk f \|_{\lqm}
		 \leq C \|f\|_{L^{p}(\R^n)}. 
	\end{equation} 
 By duality, for $q'\in[1+\frac{1}{n},\infty)$, and  $p'=\frac{nq'}{n-k}$  there exists a constant $C$ such that for all  $f\in L^{q'}(\M)$, 
	\[  \left( \int_{\R^n}|T^*_{n,k}g(x)|^{p'}\;dx \right)^{1/p'}
		 \leq C \|g\|_{L^{q'}(\M)}. \] 
 A function $f$ is a an {\it extremizer}  of the inequality \eqref{LpLq} for a specific pair $(p,q)$ if  $\|f\|_{L^p(\R^n)}$ is positive and finite, and
		\[ \frac{\|\Tnk f\|_{\lqm} }{\|f\|_{L^p(\R^n)}}= \sup_{\{g:\|g\|_{L^p(\R^n)}\neq0\}} \frac{ \|\Tnk g\|_{\lqm}}{\|g\|_{L^p(\R^n)}} . 
		\] 
	In this paper we are concerned with critical points of the functional 
\begin{equation}\label{func} \Phi(f) = \frac{\|\Tnk f\|_{\lqm} }{\|f\|_{L^p(\R^n)}}. 
\end{equation} 
  For $q\in(1,n+1]$, the nonnegative critical points of this functional satisfy the Euler-Lagrange equation 
\begin{equation}	\label{EL}
		f=\lambda(\Tnk^*[(\Tnk f)^{q_{el}}])^{p_{el}}
\end{equation}
where  $q_{el}=q-1$, $p_{el}=\frac{1}{p-1}$, and $\lambda=\left(\|f\|_{L^p(\R^n)}^{p}\|\Tnk f\|_{L^q(\M)}^{-q}\right)^{p_{el}}$. When the value of $\lambda$ is unimportant, we will refer to equation \ref{EL} as a generalized Euler-Lagrange equation.

To discuss smoothness of functions we define $\derivsx{s}$ to be the Fourier-multiplier operator defined by $\derivsx{s}f(x) = (|\xi|^s\hat{f}(\xi))^\vee$ and, analogously, $\derivsy{s} f (\theta,y)= (|\eta|^s\hat{f}(\theta,\eta))^\vee$ where the Fourier transform is taken only in the $y$ variable.

\subsection*{Results}

	\begin{thm}\label{main}  
		Let $n\geq 2$ and $\lambda\in\R$. Take $q\in(1,n+1]$, $p=\frac{nq}{n-k+kq}$, such that both $q-1$ and $\frac{1}{p-1}$ are integers. Let $f\in L^{p}(\R^n)$ be any real-valued solution of the generalized Euler-Lagrange equation \eqref{EL}.  Then $f\in C^\infty$, all partial derivatives of $f$ are bounded, and there exists $\delta>0$ such that $(1+|x|^2)^\delta \derivsx{s}f(x)\in L^{p}(\R^n)$ for all $s\geq 0$.  
	\end{thm} 

This immediately implies Theorem \ref{thm:justextremizers}. If $f\in\lp$ is an extremizer of \eqref{main}, then as $|\Tnk f| \leq \Tnk|f|$, $|f|$ is also an extremizer and further $f=c|f|$. Thus $|f|$ is a solution of \eqref{EL} to which we may apply Theorem \ref{main} whenever both $q-1$ and $\frac{1}{p-1}$ are integers.

 When $q_{el}=q-1$, $p_{el}=\frac{1}{p-1}$ are odd integers all complex-valued critical points of the functional \eqref{func} satisfy the Euler-Lagrange equation \eqref{EL},  if the powers of the complex numbers on the right hand side of \eqref{EL} are interpreted appropriately (see \cite{CQ}). Specifically, if $z\in\C$ and $0\neq s\in\R$, $z^s$ is interpreted as $z|z|^{s-1}$. When $s$ is an odd integer, $s-1$ is even, and $|z|^{s-1}$ can be written as a product of positive integer powers of $z$ and $\bar{z}$. This allows the argument for Theorem \ref{main} to be carried out for complex-valued functions with straightforward modifications to the formulas to account for various complex conjugations, but only when both $q_{el}$ and $p_{el}$ are odd integers. Again, we note that  there exists at least one choice of $n,k$ and $q$, such that $q_{el}$ and $p_{el}$ are any specified pair of odd integers other than $(1,1)$.  
 
\begin{rmk} Theorem 1 holds also for complex-valued solutions of the generalized Euler-Lagrange equation \eqref{EL} when both $q_0-1$ and $\frac{1}{p_0-1}$ are odd integers. 
\end{rmk}

The author thanks her thesis advisor Michael Christ, for his guidance at the early stages of this project and his comments on an early version of this manuscript. 

\subsection*{Notation} For the duration of the paper,  $q_0$ and $p_0$ will satisfy $q_0\in(1,n+1]$ and $p_0=\frac{nq_0}{n-k+kq_0}$ and the dimension will satisfy $n\geq 2$. Additionally, $q_{el}$ and $p_{el}$ will denote the exponents in the associated Euler Lagrange equation, namely $q_{el}=q_0-1$ and $p_{el}=\frac{1}{p_0-1}$. In Sections 4 through 7, we require that both $q_{el}$ and $p_{el}$ be integers.

In the body of the paper, $C$ will denote a finite positive constant  whose value may vary line to line but depends only on $q_0,n$, and $k$, parameters $t,\rho,\varrho$, etc. specifying weighted norms, or a parameter $s$ specifying the number of derivatives. 

In Appendix \ref{sec:kp}, $C$ will be allowed further to depend on various specified cutoff functions. 

Finally, $\mathscr{S}(\R^n)$ will denote the Schwartz class of functions on $\R^n$.

\section {Weighted inequalities} 

In this section, we consider the mapping properties of the $k$-plane transform and its adjoint on $L^p$ spaces weighted to capture additional decay. 

Fix $p_0>1$ and $q_0$ such that inequality \ref{LpLq} holds. Let $w:\R^n\to\R$ be given by
\[ w(x) = (1+|x|^2)^{\frac{n-k}{2(p_0-1)}} = \langle x \rangle ^{p_{el}(n-k)},\] 
and $w^*:\M \to\R$ be given by 
\[ w_*(\theta,y)= (1+|y|^2)^{n/2}= \langle y \rangle^{n}.\]

Note that $w^{-1}$ is a known solution of the generalized Euler-Lagrange equation \eqref{EL}.

\begin{lem}[\cite{R04}] \label{niceweights} There exist positive finite constants $C_1$ and $ C_0$ depending only on $q_0,n$ and $k$, such that 
 \begin{align*}
\Tnk(w^{-1}) = C_0 w_*^{-1/q_{el}} \\   
T^*_{n,k}(w_*^{-1})      = C_1 w^{-1/p_{el}}.  
\end{align*} 
\end{lem} 
The lemma follows from direct computation and is a special case of the formulas given in \cite{R04} Example 2.2. 


Lemma \ref{niceweights} together with the fact that  if $0\leq f\leq g$, then $\Tnk f\leq \Tnk g$, gives that 
\[ \|(\Tnk f)\; w_*^{1/q_{el}} \|_{L^\infty(\M)} \leq C \| f\, w \|_{L^\infty(\R^n)} \]
and that similar a estimate holds for $\Tnk^*$. Complex interpolation with the estimate \eqref{LpLq} yields a family of weighted estimates. 

Let $p_0'$ and $q_0'$ be the H\"older conjugates of $p_0$ and $q_0$, and set 
\begin{align*}
    q_t=q_0/(1-t)  & \quad
    p_t=p_0/(1-t) \\
    q'_t=q_0'/(1-t)  & \quad
    p'_t=p_0'/(1-t)
\end{align*}
Be aware that $p'_t\neq(p_t)'$, however the latter will not enter our discussion.

\begin{lem}\label{WI}  There exists a constant $C$ such that for all $t\in[0,1]$, for all nonnegative functions $f:\R^n\to[0,\infty)$ and $g:\M\to[0,\infty)$, 
 \begin{align*}
 \left(\int_\M(T_{n,k} f)^{q_t}w_*^{tq_t/q_{el}}\;d\lambda_{\theta^\perp}(y)\;d\theta\right)^{1/q_t}& \leq C \left(  \int_{\R^n} f^{p_t}w^{tp_t}dx\right)^{1/p_t} \\
 \left(\int_{\R^n}(\Tnk^*g)^{p'_t}w^{tp'_t/p_{el}}\;dx\right)^{1/p'_t} & \leq C \left(  \int_\M g^{q'_t}w_*^{tq'_t}\;d\lambda_{\theta^\perp}(y)d\theta\right)^{1/q'_t} .
\end{align*} 
\end{lem} 
\begin{proof}

Consider the analytic family of operators $T_zf=w_*^{z/q_{el}}\Tnk(w^{-z}f)$ on the strip $\{z:0\leq Re(z)\leq1\}$.  If $Re(z)=0$ then $T_{z}f$ is bounded from $L^{p_0}(\R^n)$ to $L^{q_0}(\M)$ for our choice of $p_0$ and $q_0$. If $Re(z)=1$, $T_{z}f$ is bounded from $L^{\infty}(\R^n)$ to $L^{\infty}(\M)$, by Lemma \ref{niceweights}. Both bounds are uniform in $Im(z)$. Therefore, the first conclusion follows by complex interpolation.  

The proof of the second inequality is similar. 
\end{proof} 

Given this decay estimate, we define the following spaces.  

\begin{defn}\label{def:XtYt}Define the spaces $X_t$ and $Y_t$ to be the set of all equivalence classes of measurable functions on $\R^n$ for which the following weighted norms are finite: 
\begin{align*}
||f||^{p_t}_{X_t}  =&\int_{\R^n}|f|^{p_t}w^{t p_t}\;dx \\
||f||^{p_t'}_{Y_t}  =&\int_{\R^n}|f|^{p_t'}w^{t p_t'/p_{el}}\;dx 
\end{align*} 
Similarly, define $X_{*,t}, $ and $Y_{*,t}$ to be the sets of all equivalence classes of measurable functions on $\M$  for which the following weighted norms are finite: 
\begin{align*}
||g||^{q_t'}_{X_{*,t}}\!\! =&\int_{\M}|g|^{q'_t}w_*^{t q'_t}\;d\lambda_{\theta^\perp}\;d\mu \\ 
||g||^{q_t}_{Y_{*,t}}=&\int_{\M}|g|^{q_t}w_*^{t q_t/q_{el}}\;d\lambda_{\theta^\perp}\;d\mu .\\ 
\end{align*} 
\end{defn}

\begin{defn} 
For all complex valued functions  $f:\R^n\to\C$ and $g:\M\to\C$, set 
\begin{align*}
\Tel(f)=&(\Tnk f)^{q_{el}} \\
\Tel_*(g)=&(\Tnk^*g)^{p_{el}} \\
\Sel(f)=&\Tel_*(\Tel(f)) . 
\end{align*} 
\end{defn} 
With this notation, the generalized Euler-Lagrange equation  \eqref{EL} becomes 
\begin{equation*}
f=\lambda \Sel(f). 
\end{equation*}

\begin{prop}\label{sE}  There exists a constant $C$ such that for all $t\in[0,1]$, for all $f\in X_t$, 
\begin{equation*}
||\Sel f||_{X_{t}}\leq C ||f||_{X_t} ^{q_{el}p_{el}}.
\end{equation*} 
\end{prop} 

The proposition follows from the following estimates:

\begin{lem}\label{bddops} There exists a constant $C$ such that for any $t\in[0,1]$, for any functions $f\in X_t$ and $g\in X_{*,t}$ 
\begin{align*}
 ||\Tnk f||_{Y_{*,t}} \leq& C ||f||_{X_t} \\
  ||\Tnk^* f||_{Y_{t}} \leq& C ||f||_{X_{*,t}} \\
||\Tel f||_{X_{*,t}} \leq&  C ||f||_{X_t}^{q_{el}} \\
||\Tel_* g||_{X_{t}} \leq&  C ||g||^{p_{el}}_{X_{*,t}} 
\end{align*} 
\end{lem} 
\begin{proof} 

 As $|\Tnk f|\leq \Tnk|f|$, the first estimate follows immediately from  Lemma \ref{WI}. The proof of other two estimates is similar, but requires a bit more checking of numerology.  By definition, 
\begin{equation*}
||\Tel f||_{X_{*,t}}=\left(\int_{\M}|\Tnk f|^{q_{el}q'_t}w_*^{t q'_t} \right)^{1/q'_t}.  
\end{equation*}
Using that $q_{el} = q_0-1$, we find $q'_0=\frac{q _0}{q_{el}}$ and hence $q_{el}q_t'=q_t$. Again using that $|\Tnk f|\leq \Tnk |f|$ and Lemma \ref{WI}, 
\begin{equation*}
||\Tel f||_{X_{*,t}}\leq\left(\int_{\M}(\Tnk | f|)^{q_t}w_*^{t\frac{q_t}{q_{el}}} \right)^{q_{el}/q_t} \leq C \left(  \int_{\R^n} |f|^{p_t}w^{tp_t}\right)^{q_{el}/p_t}\leq C||f||_{X_t}^{q_{el}}. 
\end{equation*}

Similarly,
\begin{equation*}
||\Tel_* g||_{X_{t}} 
=\left(\int_{\R^n}|\Tnk^*g|^{p_{el}p_t}w^{t p_t}\right)^{1/p_t}.
\end{equation*}
Here, $p_{el} = \frac{1}{p_0-1}$, so $p'_0=p _0p_{el}$ and hence $p_{el}p_t=p_t'$. Again, $|\Tnk ^*g|\leq \Tnk |g|$ so by Lemma \ref{WI}, 
\begin{multline*}
||\Tel_* g||_{X_{t}}
 \leq \left(\int_{\R^n}(\Tnk^* |g|)^{p'_t}w^{t p'_t/p_el}\right)^{p_{el}/p'_t} \\
 \leq C \left(  \int_\M |g|^{q'_t}v^{-tq'_t}\right)^{p_{el}/q'_t}=  C ||g||_{X_{*,t}}^{p_{el}}. \qedhere
\end{multline*}
\end{proof} 

Finally we record some properties of these weighted spaces, which are direct consequences of the definitions and H\"older's inequality. 
\begin{lem}\label{littlethings} The following statements hold for the spaces $X_t$.  They hold as well when $X_t$ is replaced by $Y_t$, $X_{*,t}$ or $Y_{*,t}$. 
\begin{enumerate}
\item If $\alpha<\beta$ then $X_\beta\subset X_\alpha$. In particular, there exists a constant $C$ such that for all $0\leq\alpha\leq\beta\leq 1$,  for all $f\in X_\beta$, 
\[ \|f\|_{X_\alpha}\leq C \|f\|_{X_\beta}.\] 
\item Let $0\leq \alpha\leq\gamma\leq \beta <1 $. Let $\gamma=\theta\alpha+(1-\theta)\beta$. Then, 
\[  \|f\|_{X_{\gamma}}\leq  \|f\|_{X_{\alpha}}^\theta \|f\|_{X_{\beta}}^{1-\theta}\] 

\end{enumerate} 
\end{lem}
  Let $A_p$ denote the standard Muckenhoupt classes of weights. Whenever $w^{tp_t}\in A_{p_t}$ operators of Calder\'on-Zygmund type are bounded on $L^{p_t}(w^{tp_t})=X_t $ (see for instance \cite{Stein}, pg 205). 

\begin{lem}\label{ApWeights} For all sufficiently small $t$, the weight $w= (1+|x|^2)^{\frac{p_{el}(n-k)}{2}}$  satisfies $w^{tp_t}\in A_{p_t}(\R^n)$.
\end{lem} 
\begin{proof} 
For any polynomial $u$, $|u|^s\in A_p$ if $-1<s(deg(u))<p-1$ (\cite{Stein}, pg. 219). Thus for $t\geq0$  $w^{tp_t}\in A_{p_t}$ if  
\[ p_{el}(n-k)tp_t< p_t-1.\]
  Using that $p_t=\frac{p_0}{1-t}$, this holds whenever
\[ 
 t<\frac{(p_0-1)^2}{p_0(n-k-1)+1}.
\] 
As $p_0>1$  and $1\leq k\leq n-1$ this bound is strictly positive.

\end{proof} 

We will also need similar results for the weights $w_*$, adapted to the spaces $Y_{*,t}$ and $X_{*,t}$. 

\begin{lem} \label{ApWeightsThetaPerp} Let $ w_*(\theta,y)=\langle y \rangle^{n}$.  For all sufficiently small $t$, for every $\theta\in \mathcal{G}_{n,k}$, $w_*^{tq'_t}(\theta,y) \in A_{q'_t}(\theta^\perp)$ and  $w_*^{tq_t/q_{el}}(\theta,y) \in A_{q_t}(\theta^\perp)$. 
\end{lem} 
\begin{proof} 
 Using again that for any polynomial $u$, $|u|^s\in A_p$ if $-1<s(deg(u))<p-1$, for $t\geq0$  $(w_*)^{tq'_t}\in A_{q_t'}$ if $ntq_t'< q'_t-1$. Using that $q'_t=\frac{q_0}{(q_0-1)(1-t)}$, gives that  $(w_*)^{tq'_t}\in A_{q_t'}$ if 
\[t < \frac{1}{ q_0 n- q_0+1}.\] 
  
Secondarily, $(w_*)^{tq_t/q_{el}}\in A_{q_t}$ if $nt\frac{q_t}{q_{el}}< q_t-1$. Using that $q_t=\frac{q_0}{1-t}$  and $q_{el}=q_0-1$ 
 gives that  $(w_*)^{tq_t/q_{el}}\in A_{q_t}$ if 
\[ t< \frac{(q_0-1)^2} {q_0(n-1)+1 }. \] 
Again note that each of these bounds is strictly positive. 

\end{proof}

\section{Smoothing} 
The $k$-plane transform of a function benefits from additional smoothness in the translation variable $y$. Recall that for $\gamma\in\R$, 
$\derivsy{\gamma} $ denotes the Fourier multiplier for functions on $\M$,  defined by $\widehat{\derivsy{\gamma} g(\theta,y)} = |\eta|^\gamma\widehat{g}(\theta,\eta)$ where the Fourier transform is taken only in the translation variable $y.$

\begin{lem}[Strichartz \cite{S81}\footnote{Strichartz actually proves a stronger mixed norm estimate. This follows from his, using H\"older's inequality and the observation that $p\leq 2$ ensures that $p'\geq p$. %
}] \label{StrichartzEstimate} 
For $p\in(1,2]$, there exists a constant $C$ such that 
\[ \| \derivsy{k/p'}  \Tnk f\|_{L^p(\M)}\leq C \|f\|_{L^p(\R^n)}. \] 
\end{lem} 

In our context, some derivative of the $k$-plane transform of a function in the weighted space $X_t$, lives in $L^{q_0}(\M)$. 
	
\begin{lem}\label{smoothing} Let $q_0\in(1,n+1]$ such that $p_0=\frac{nq_0}{n-k+kq_0}$  satisfies $p_{0}\in(1,2]$.  Let $t>0$.  There exists $\gamma=\gamma(t)>0$ such that for any $f\in X_t$, $\derivsy{\gamma} ( \Tnk f)\in L^{q_0}(\M)$.  In particular, there exists a positive finite constant $C$ such that for any $f\in X_t$, 
\[ \| \derivsy \gamma( \Tnk f) \|_{L^{q_0}(\M)} \leq C \| f\|_{X_t} . \]
\end{lem}  

\begin{proof} 

Let $t>0$. Let $f\in X_t$. From Lemma \ref{bddops}, $\| \Tnk f\|_{Y_{*,t}}\leq C\|f\|_{X_t}$. It follows from the definition of  $Y_{*,t}$ that there exists $r>q_0$ such that the space   $Y_{*,t}$ embeds continuously into $L^{r}(\M)$. Thus by Lemma \ref{bddops} for this $r>q_0$, 
\[  
 \| \Tnk f\|_{L^{r}(\M)}\leq C\|f\|_{X_t}.
\] 

As $p_0\in(1,2]$, Lemma \ref{StrichartzEstimate} and Lemma \ref{littlethings} yields
\[ 
\| \derivsy{k/p'}  \Tnk f\|_{L^{p_0}(\M)}\leq C \|f\|_{L^{p_0}(\R^n)}\leq C \|f\|_{X_t}.
\] 
Using the analytic family of operators $\derivsy{zk/p'} \Tnk$ to interpolate between these two estimates  gives that for $\theta\in[0,1]$, 
\[ \| \derivsy{\theta k/p'}  \Tnk f\|_{L^{Q(\theta)}(\M)}\leq C \| f\|_{X_t} \] 
where $Q(\theta)^{-1}=\frac{1}{p_0}\theta+\frac{1}{r}(1-\theta)$. As $q_{0}>1$, $p_0=\frac{n}{n-k+kq_0}(q_0)<q_0$. Therefore there exists $\theta\in(0,1)$, such that $Q(\theta)=q_0$. 

\end{proof}

\section{Multilinear Bounds}
For the rest of the paper we require that $p_0$ and $q_0$ are chosen such that $p_{el},q_{el}\in\N$.  
\begin{defn} Let $ \vec{f}=\{f_{i,j}\} \text{  for  } i\in[1,p_{el} ], j\in[1,q_{el} ] $. Define the multilinear operator $\vec{S}$ by: 
\[ \vec{S}(\vec{f})= \prod_{i=1}^{p_{el}} T^*_{n,k}\left( \prod_{j=1}^{q_{el}}\Tnk (f_{i,j})\right) . \] 
\end{defn} 
Thus $\Sel(f)=\vec{S}(f,\ldots,f)$.
\begin{lem}\label{basicmle}  For each $ \vec{f}=\{f_{i,j}\}$  for  $i\in[1,p_{el} ]$ and $ j\in[1,q_{el} ] $,
\[ |\vec{S}(\vec{f})|\leq \prod_{i=1}^{p_{el}}\prod_{j=1}^{q_{el}}\Sel(|f_{i,j}|)^{\frac{1}{p_{el}q_{el}}} . \] 
\end{lem} 
\begin{proof} 
As both $|\Tnk^* (g)| \leq \Tnk^*(|g|)$ and $|\Tnk(f)| \leq \Tnk(|f|)$, 
\[ |\vec{S}(\vec{f}) | \leq   \prod_{i=1}^{p_{el}}  T^*_{n,k}\left( \prod_{j=1}^{q_{el}}\Tnk(|f_{i,j}|)\right). \]
Also, 
\[  T^*_{n,k}\left( \prod_{j=1}^{q_{el}}g_j \right)= \int_{\mathcal{G}_{n,k}} \prod_{j=1}^{q_{el}}g_j(\theta,P(x,\theta^\perp))\; d\theta \] 
By repeated applications of H\"older's inequality, 
\[  \Tnk^* \left( \prod_{j=1}^{q_{el}}g_j \right)\leq \prod_{j=1}^{q_{el}} \left( \int_{\G} g_j(\theta,P(x,\theta^\perp))^{q_{el}}\;d\theta \right)^{1/q_{el}} \leq  \prod_{j=1}^{q_{el}} (\Tnk^*(g_j^{q_{el}}) )^\frac{1}{q_{el}}.  \]
Applying this for each $i$ with $ g_j=\Tnk (|f_{i,j}|)$  gives 
\[ |\vec{S}(\vec{f}) | \leq  \prod_{i=1}^{p_{el}}  \prod_{j=1}^{q_{el}}  \Tnk^* \left([\Tnk(|f_{i,j}|)]^{q_{el}} \right)^\frac{1}{q_{el}} =   \prod_{i=1}^{p_{el}}\prod_{j=1}^{q_{el}}\Sel(|f_{i,j}|)^{\frac{1}{p_{el}q_{el}}} .\] 
\end{proof} 

 Lemma \ref{basicmle} and a weighted multilinear version of H\"older's inequality (Christ and Xue's Lemma 3.1, \cite{CX}) combine to give multilinear estimates for $\Sel$ which will be used in the following sections. 

\begin{lem}[Christ and Xue, \cite{CX}] \label{lemCX} Let $p_0\in[1,\infty)$, $t>0$, and $p_t=\frac{p_0}{1-t}$. Let $X_t\subset L^p(\R^n)$ with norm given by $||f||^{p_t}_{X_t}=\int_{\R^n}|f|^{p_t}w^{t p_t}$ for some measurable function $w\geq 1$. Let $A$ be any finite index set. Let $\theta_\alpha,t_\alpha\in[0,1]$ for each $\alpha\in A$. Suppose that $\sum_{\alpha\in A}\theta_\alpha=1$ and $1-t=\sum_{\alpha\in A} \theta_\alpha(1-t_\alpha)$.  Then for any nonnegative functions $\{f_\alpha\}_{\alpha\in A}$, 
\[ ||\prod_{\alpha\in A} f_\alpha^{\theta_\alpha}||_{X_t}\leq \prod_{\alpha\in A}|| f_\alpha||_{X_{t_\alpha}}^{\theta_\alpha} .\]
\end{lem} 

This Lemma is stated in \cite{CX} for the specific weights considered there, but the proof holds in the generality stated above. 

\begin{lem}\label{mlemult}  Let $t\in[0, 1 ]$.  Let $ A =\{ (i,j):  0\leq i \leq q_{el},0\leq j \leq p_{el} \}$.  Let $\{ f_{\alpha}: \alpha\in A\}$ satisfy $f_\alpha\in X_t$ for all $\alpha\in A$. Then $\vec{S}(\vec{f})\in X_t$ and 
\[ \|\vec{S}(\vec{f})\|_{X_t} \leq C \prod_{\alpha\in A} \|f_\alpha\|_{X_t}. \]
\end{lem}
\begin{proof}
Estimating $\vec{S}$ in terms of $\Sel$ by Lemma \ref{basicmle},
\[ \| |\vec{S}(\vec{f})\|_{X_t}\leq \| \prod_{\alpha\in A} \left( \Sel (|f_\alpha|)\right)^{\frac{1}{p_{el}q_{el}}} \|_{X_t} .\] 
Applying Lemma \ref{lemCX} with $\theta_{\alpha}=\frac{1}{p_{el}q_{el}}$ and $t_\alpha=t$ for each $\alpha\in A$, 
\[  \| \prod_{\alpha\in A} \left( \Sel |f_\alpha|\right)^{\frac{1}{p_{el}q_{el}}} \|_{X_t}
\leq 
 \prod_{\alpha\in A}|| \Sel | f_\alpha| ||_{X_{t}}^{\frac{1}{p_{el}q_{el}}}. 
 \] 
The lemma then follows from the bound on $\Sel$ from Proposition \ref{sE}. 

 \end{proof} 

\begin{lem}\label{mleonebetter}  Let $t\in[0, \frac{1}{p_{el}q_{el}} ]$.  Let $ A =\{ (i,j):  0\leq i \leq q_{el},0\leq j \leq p_{el} \}$.  Let $\{ f_{\alpha}: \alpha\in A\}$ such that  $f_\alpha\in X_0$ for all $\alpha\in A$ and suppose further that there exists $\beta\in A$ such that $f_\beta\in X_{tp_{el}q_{el}}$. Then 
\[ \|\vec{S}(\vec{f})\|_{X_t} \leq C \left(  \|f_\beta\|_{ X_{tp_{el}q_{el}}}\prod_{\alpha\neq\beta} \|f_\alpha\|_{X_0}\right). \]
\end{lem}

\begin{proof}
Again, estimating $\vec{S}$ in terms of $\Sel$ by Lemma \ref{basicmle},
$$\| |\vec{S}(\vec{f})\|_{X_t}\leq \| \prod_{\alpha\in A} \left( \Sel (|f_\alpha|)\right)^{\frac{1}{p_{el}q_{el}}} \|_{X_t} .$$
Apply Lemma \ref{lemCX} with $\theta_{\alpha}=\frac{1}{p_{el}q_{el}},$ $t_\alpha=0$ for $\alpha\neq\beta$, and $t_\beta=p_{el}q_{el}t$ to obtain 
\[  \| \prod_{\alpha\in A} \left( \Sel (|f_\alpha|)\right)^{\frac{1}{p_{el}q_{el}}} \|_{X_t}
\leq 
 \prod_{\alpha\in A}|| \Sel ( |f_\alpha|) ||_{X_{t_\alpha}}^{\frac{1}{p_{el}q_{el}}}. 
 \] 
Using the bound on $\Sel$ from Proposition \ref{sE} and that  $t_\alpha=0$ for $\alpha\neq\beta$ and $t_\beta=p_{el}q_{el}t$ , 
\[  \prod_{\alpha\in A}|| \Sel( |f_\alpha|) ||_{X_{t_\alpha}}^{\frac{1}{p_{el}q_{el}}}
\leq C \prod_{\alpha\in A}|| f_\alpha ||_{X_{t_\alpha}}\leq C || f_\alpha ||_{X_{p_{el}q_{el}t}}\prod_{\alpha\neq\beta}|| f_\alpha ||_{X_{0}}. \qedhere
 \] 
\end{proof}

\section {Extra Decay}

In this section we show that any solution of the generalized Euler-Lagrange equation (\ref{EL}) while initially only assumed to be in $L^p$ in fact has better decay:

\begin{prop}\label{extraDecay} Let $q_0\in(1,n+1]$ such that for $p_0=\frac{ nq_0}{n-k+kq_0}$, $q_{el}=q_0-1$ and $p_{el}=\frac{1}{p_0-1}$ are both integers.   Let $n\geq 2$ and $\lambda\in\R$.    Let $f\in L^{p_0}(\R^n)$ be a real-valued solution of the generalized Euler-Lagrange equation $f=\lambda \Sel f.$ Then there exists $t>0$ such that $f\in X_t$. 
\end{prop} 
The proof given here is essentially the same as that  of Proposition 5.1 in \cite{CX}.\\
If $f$ is a solution of the generalized Euler-Lagrange equation, and $f=g+\varphi$ is some decomposition of $f$, then $g$ solves the generalized Euler-Lagrange equation upto an error given by 
\begin{align*} 
 \mathcal{L}(\varphi, g) &= g - \lambda S(g) \\
 &= \lambda \Sel(\varphi + g) - \lambda \Sel(g) - \varphi. 
\end{align*} 

This operator is well-defined for $\varphi,g\in X_0$. We require further than $\varphi\in L^\infty(\R^n)$.

 For each $\epsilon>0$, chose a decomposition of $f$, $f=\varphi_\epsilon+g_\epsilon$, such that $\|g_\epsilon\|_{X_0}<\epsilon$, $\varphi_\epsilon$ has bounded support, and $\varphi_\epsilon\in L^\infty$. Define 
\[ A_\epsilon(h) =\lambda \Sel (h )  + \mathcal{L}(\varphi_\epsilon, g_\epsilon). \] 
 As $f$ is a solution of the generalized Euler-Lagrange equation, $g_\epsilon$ is a solution of $A_\epsilon(h)=h$ in the space $X_0$. Proposition \ref{extraDecay} follows from showing that $g_\epsilon$ in fact has better decay. The main step is the following lemma.

\begin{lem} \label{fixedpoint} 
 Let $q_0\in(1,n+1]$ such that for $p_0=\frac{ nq_0}{n-k+kq_0}$, $q_{el}=q_0-1$ and $p_{el}=\frac{1}{p_0-1}$ are both integers.   Let $n\geq 2$ and $\lambda\in\R$.    Let $f\in L^{p_0}(\R^n)$ be a real-valued solution of  $f=\lambda \Sel(f)$.  For each $\epsilon>0$, let  $f=\varphi_\epsilon+g_\epsilon$ be any decomposition such that $\|g_\epsilon\|_{X_0}<\epsilon$ and $\varphi_\epsilon\in L^\infty$ has bounded support. Then there exists $\epsilon_0>0$ such that for each $\epsilon\in(0,\epsilon_0]$, there exists $t_\epsilon>0$ such that for all $t\in[0,t_\epsilon]$, $A_\epsilon:B_t(0,\epsilon^{1/2})\to B_t(0,\epsilon^{1/2})$ is a strict contraction and the fixed point equation $A_\epsilon(h)=h$ has a unique solution $h\in X_t$ satisfying $\|h\|_{X_t}\leq \epsilon^{1/2} $.
\end{lem} 

We first give the deduction of Proposition \ref{extraDecay} from Lemma \ref{fixedpoint}. The key idea is that while the solution $h\in X_t$ guaranteed by Lemma \ref{fixedpoint} might in principle depend on $t$, because the spaces are nested, it does not. 

\begin{proof}[Proof of Proposition \ref{extraDecay}] 
Let $\epsilon_0$ be the small quantity guaranteed by Lemma \ref{fixedpoint}. Fix $\epsilon\in(0,\epsilon_0]$ and let $0\leq s\leq t\leq t_\epsilon$.  Let $h\in X_s$, $\|h\|_{X_s} \leq \epsilon^{1/2}$ and $A_\epsilon(h)=h$.  Similarly, let $\tilde{h}\in X_t$, $\|\tilde{h}\|_{X_t} \leq \epsilon^{1/2}$ and $A_\epsilon(\tilde{h})=\tilde{h}$.  As $\tilde{h}\in X_t$ and $s\leq t$, by Lemma \ref{littlethings} $\tilde{h}\in X_s$ and 
\[ \| \tilde{h} \|_{X_s} \leq  \|\tilde{h}\|_{X_t} \leq \epsilon^{1/2} . \]
Therefore, by uniqueness of solutions in $X_s$,  $h=\tilde{h}$. 

Taking $s=0$, by definition $h=g_\epsilon$ is a solution of $A_\epsilon(h)=h$, and this uniqueness implies that $g_\epsilon\in X_t$ for all $t\in[0,t_\epsilon]$ for all sufficiently small $\epsilon$. Thus, $f=g_\epsilon +\varphi_\epsilon$, also satisfies $f\in X_t$. 
\end{proof} 

\begin{proof}[Proof of Lemma \ref{fixedpoint}]   
We begin by estimating $\|\mathcal{L}(\varphi_\epsilon, g_\epsilon)\|_{X_t}$ for small $t$.  
First, as  $g_\epsilon+\varphi_\epsilon$ is a solution to the generalized Euler-Lagrange equation  \eqref{EL}, there exists $\lambda\in\C$ such that $ g_\epsilon+\varphi _\epsilon= \lambda \Sel(g_\epsilon+\varphi_\epsilon) $. 
Therefore, 
\[  \mathcal{L}(\varphi_\epsilon, g_\epsilon) = g_\epsilon-\lambda \Sel g_\epsilon. \]
By the bound on $\Sel$ from Proposition  \ref{sE} and the triangle inequality for $\epsilon\leq1$,  
\[ \|\mathcal{L}(\varphi_\epsilon, g_\epsilon)\|_{X_0} \leq \| g_\epsilon \|_{X_0}+C \| g_\epsilon\|_{X_0}^{p_{el}q_{el}} \leq C\epsilon . \] 
Next consider $t =\frac{1}{p_{el}q_{el}}$ the largest value of $t$ for which the multilinear estimate Lemma \ref{mleonebetter} applies. Working directly from the definition of $\mathcal{L}$, 
\[ \|\mathcal{L}(\varphi_\epsilon, g_\epsilon) \|_{X_{1/p_{el}q_{el}}} \leq |\lambda| \|  \Sel(\varphi_\epsilon + g_\epsilon) - \Sel(g_\epsilon))  \|_{X_{1/p_{el}q_{el}}}+ \| \varphi_\epsilon  \|_{X_{1/p_{el}q_{el}}}. \] 
Expand $\Sel(\varphi_\epsilon + g) - \Sel(g_\epsilon)$  as a sum of $p_{el}q_{el}-1$ terms each of the general form $\vec{S}(\vec{f}_i)$ where  $\vec{f_i}=(f_{i,\alpha}:  \alpha\in A)$ and  $f_{i,\alpha}\in\{\varphi_\epsilon,g_\epsilon\}$ and for each term there is at least one index $\beta$ such that $f_{i,\beta}=\varphi_\epsilon$.   Applying Lemma \ref{mleonebetter} to each such term gives, 
\[ \|\mathcal{L}(\varphi_\epsilon, g_\epsilon) \|_{X_\frac{1}{p_{el}q_{el}}} \leq C \sum_{i=1}^{p_{el}q_{el}} \left( \|\varphi_\epsilon\|_{X_1}\prod_{\alpha\neq\beta} \|f_{i,\alpha}\|_{X_0} \right) 
+ \| \varphi_\epsilon  \|_{X_\frac{1}{p_{el}q_{el}}}. \] 
	As $ \|g_\epsilon\|_{X_0} <\epsilon$, and for each $s\in[0,1]$,  $ \|\varphi_\epsilon\|_{X_s} <\|\varphi_\epsilon\|_{X_1}$, there exists a finite constant depending on $\varphi_\epsilon$ such that 
\[ \|\mathcal{L}(\varphi_\epsilon, g_\epsilon) \|_{X_{1/p_{el}q_{el}}} \leq C_{\varphi_\epsilon} . \] 
By Lemma \ref{littlethings}, in particular, by convexity of the $X_t$ norms, for sufficiently small $\epsilon>0$, there exists $t_\epsilon>0$, such that for each $t\in(0,t_\epsilon]$, 
\[ \|\mathcal{L}(\varphi_\epsilon, g_\epsilon) \|_{X_{t}}\leq \epsilon^{3/4}. \] 

Consider now bounds for $A_\epsilon$.   By the triangle inequality, 
\[ \| A_\epsilon(h) \|_{X_{t}} \leq |\lambda|\|\Sel h\|_{X_{t}}+  \|\mathcal{L}(\varphi_\epsilon, g_\epsilon) \|_{X_{t}} . \] 
Using the bound on $\Sel$ from Proposition  \ref{sE}, 
\[ \| A_\epsilon(h) \|_{X_{t}} \leq C \| h\|_{X_{t}}^{p_{el}q_{el}} +  \|\mathcal{L}(\varphi_\epsilon, g_\epsilon) \|_{X_{t}} . \] 
Let $B_t(0,\epsilon^{1/2})$ be the open ball of radius $\epsilon^{1/2}$ centered at $0$ in $X_{t}$. If $t\in(0,t_\epsilon]$, for $h\in B_t(0,\epsilon^{1/2})$,
\[ 
 \| A_\epsilon(h) \|_{X_{t}}  \leq  C\epsilon^{p_{el}q_{el}/2} + \epsilon^{3/4} . 
\]
For sufficiently small $\epsilon$ it follows that  $ \| A_\epsilon(h) \|_{X_{t}} < \epsilon^{1/2}$. Thus for sufficiently small $\epsilon$, for every $t\in(0,t_\epsilon]$, 
\[  A_\epsilon( B_t(0,\epsilon^{1/2}))\subset B_t(0,\epsilon^{1/2}). \] 
Consider $\tilde{h}, h\in  B_t(0,\epsilon^{1/2})$. 
\[ \| A_\epsilon(h) -A_\epsilon(\tilde{h})  \|_{X_t} = C \| \Sel h - \Sel \tilde{h} \|_{X_t} . \]  
Write out $\Sel (h)$ and $\Sel(\tilde{h})$ in terms of the multilinear operator $\vec{S}$. Adding and subtracting terms of the form $\vec{\Sel}(h,\ldots,h,\tilde{h}, \tilde{h},\ldots,\tilde{h})$, allows one to write $\Sel (h) - \Sel( \tilde{h})$ 
as a sum of $p_{el}q_{el}$ terms, where each term is of the general form $\vec{S}(\vec{f})$ with $\vec{f}=(f_\alpha: \alpha\in A)$  such that there is one index $\beta$ such that $f_\beta=h-\tilde{h}$, and for  $\alpha\neq\beta$, $f_\alpha$ is either $h$ or $\tilde{h}$.    Applying the multilinear estimate from Lemma \ref{mlemult} to each such term gives that there exists $C$ independent of $\epsilon$ such that 
\[ \| A_\epsilon(h) -A_\epsilon(\tilde{h})  \|_{X_t} \leq C \epsilon^{(p_{el}q_{el}-1)/2}  \| h - \tilde{h} \|_{X_t} .\]  
Thus when $\epsilon$ is sufficiently small $A_\epsilon:B_t(0,\epsilon^{1/2})\to B_t(0,\epsilon^{1/2})$ is a strict contraction.  Therefore, there exists a unique $h_\epsilon\in X_t $ such that $\|h_\epsilon\|_{X_t}\leq \epsilon^{1/2}$ and $A_\epsilon(h_\epsilon)=h_\epsilon$.

\end{proof}

\section {Mollified derivatives} 

The next goal is to show that if a solution of the Euler-Lagrange equation has additional decay, then its derivatives exist and behave well. Following Christ and Xue \cite{CX}, we use mollified derivatives to prove the smoothness result. In this section we state the definition of these derivatives and give their key properties, which we'll use in Section \ref{sec:conclusion} to complete the proof of Theorem \ref{main}.  We correct a small technical error from \cite{CX} in the definition of these mollified derivatives.

 Recall that $\langle x\rangle = (1+|x|^2)^{1/2}$ and $\mathscr{S}(\R^n)$ denotes the Schwartz class of functions on $\R^n$. Let $\mathscr{S}(\M)$ denote the class of functions on $\M$, satisfying ``$f(\theta,y)\in\mathscr{S}(\theta^\perp)\sim\mathscr{S}(\R^{n-k})$, uniformly in $\theta$".  We call this class the Schwartz class of functions on $\M$.

\begin{defn}\label{derivdef}
For all $f\in\mathscr{S}(\R^n)$ , for each $s\geq0$ and $\Lambda\geq1$ define the operator $\derivslx{s}{\Lambda}$ by 
\[ \widehat{\derivslx{s}{\Lambda}f}(\xi)=\frac{\langle \xi \rangle ^s}{\langle\Lambda^{-1}\xi \rangle ^s} \hat{f}(\xi).  \]
Similarly, for  all   $g(\theta,y)\in \mathscr{S}(\M)$, for each $s\geq0$ and $\Lambda\geq1$ define $\derivsly{s}{\Lambda} $ by 
\[  \widehat{\derivsly{s}{\Lambda}g_\theta}(\eta) =\frac{\langle \eta \rangle ^s}{\langle\Lambda^{-1}\eta \rangle ^s} \widehat{g_\theta}(\eta).   \] 
where the Fourier transform is taken only in the $y$ variable.
\end{defn} 

The first observation is that for fixed $\Lambda$ the mollified derivative retains some decay. 

\begin{lem}\label{niceMD}  For all sufficiently small $\rho\geq 0$, for each $\Lambda>0$, there exists a constant $C_{\Lambda}$, such that for all $f\in X_\rho$, 
\[ \|\derivslx{s}{\Lambda}f \|_{X_\rho} \leq C_\Lambda \|f\|_{X_\rho}. \] 
\end{lem} 

\begin{proof} 
For sufficiently small $\rho$, $w^{\rho p}\in A_p$ by Lemma \ref{ApWeights}.  
By the H\"ormander-Mihlin multiplier theorem (see e.g \cite{Stein} pg 26, 205) it is enough to check that for each multi-index $\alpha$ such that $|\alpha|\leq\frac{n}{2}+1$,  $ \left| \partial^\alpha_\xi \left(\frac{\langle \xi\rangle^s}{\langle\Lambda^{-1} \xi\rangle^s} \right) \right| \leq \frac{C_{\alpha}}{|\xi|^{|\alpha|}}$ .  
Direct computation shows that for $\alpha$ satisfying $|\alpha|\leq\frac{n}{2}+1$,  
 \begin{equation}\label{Lambdas}  \left| \partial^\alpha_\xi \left(\frac{\langle \xi\rangle^s}{\langle\Lambda^{-1} \xi\rangle^s} \right) \right| = \Lambda^{s} \left| \partial^\alpha_\xi \left(\frac{1+|\xi|^2}{\Lambda^2+|\xi|^2} \right)^{s/2} \right|
\leq \frac{C_{s,\alpha}\Lambda^{s+2}}{(\Lambda^2+|\xi|^2)^{\frac{|\alpha|+2}{2}}}
\leq \frac{C_{s,\alpha}\Lambda^{s}}{|\xi|^{|\alpha|}}.
\end{equation} 
We note that the final constant $C_\Lambda$ is allowed to depend on $s.$

\end{proof}

 Moreover, these operators intertwine through the $k$-plane and the dual $k$-plane transforms.

\begin{lem}\label{commute}  For all $f\in\mathscr{S}(\R^n)$ and $g\in\mathscr{S}(\M)$ the following formulas hold:
\[  \derivsly{s}{\Lambda} \Tnk f = \Tnk \derivslx{s}{\Lambda} f \] 
\begin{equation}\label{Tsdls}  \derivslx{s}{\Lambda} \Tnk^* g = \Tnk^* \derivsly{s}{\Lambda} g.
\end{equation} 

\end{lem} 

The key observation in the proof is that $\Tnk$ and $\Tnk^*$ interact nicely with the Fourier transform.  It is well known (e.g. \cite{SmS75}) that for all  $f\in\mathscr{S}(\R^n)$, for all  $\theta\in\mathcal{G}_{n,k}$, for all $\xi\in\theta^\perp$
\[ \widehat{ \Tnk f _\theta}(\xi) = c_{n,k}\hat{f}(\xi) \]
where $c_{n,k}$ is a constant that depends on the normalization chosen for the integrals appearing in the definitions of the operators. 
 
 A direct computation shows that a similar formula holds for the adjoint as well. 
\begin{lem}\label{DualF}
 For all $g\in\mathscr{S}(\M)$, for all  $\theta\in\G$, for all $\xi\in\theta^\perp$
\[ \widehat{ \Tt g }(\xi) = \mathscr{C}_{n,k}\int_{\{ \theta: \theta\perp\xi\}}  \widehat{g}(\theta,\xi)\; d\gamma_{\xi^\perp}(\theta) \]
where for functions on $\M$ the Fourier transform is taken only in the $y$-variable and $d\gamma_{\xi^\perp}$ represents the restriction of the measure $d\gamma$ to the subset of $k$-planes which are perpendicular to $\xi$ and $\mathscr{C}_{n,k}$ is a constant that depends on the normalization chosen for the integrals appearing in the definitions of the operators. 
\end{lem} 

\begin{proof} Using the definitions and Fubini's theorem, 
\begin{align*} \widehat{ \Tt g }(\xi) =&C \int_{\R^n} \Tt g(x) e^{-2\pi i x\cdot \xi}\;dx \\
 =& C\int_{\R^n}\int_{\G}g(\theta,P_{\theta^\perp}(x)) e^{-2\pi i x\cdot \xi}d\gamma(\theta)\;dx \\
 =&C\int_{\G} \int_{\theta}\int_{\theta^\perp}g(\theta,x_1) e^{-2\pi i (x_1\cdot \xi_1+x_2\cdot \xi_2)}\;d\lambda_{\theta^\perp}(x_1)\;d\lambda_{\theta}(x_2)d\gamma(\theta) 
\end{align*}
where $x_1=P_{\theta^\perp}(x)$ and $x_2=P_{\theta}(x)$ and similarly for $\xi$. 
Further, writing $\delta_\theta(\xi_2)$ for the distribution that indicates that $\xi_2\in\theta$ is $0$, 
\begin{align*} \widehat{ \Tt g }(\xi)  =&C\int_{\G} \int_{\theta}\left(\int_{\theta^\perp}g(\theta,x_1)e^{-2\pi i (x_1\cdot \xi_1)}\;d\lambda_{\theta^\perp}(x_1)\right)  e^{-2\pi i (x_2\cdot \xi_2)}\;d\lambda_{\theta}(x_2)\;d\gamma(\theta) \\
=&C\int_{\G}\hat{g}(\theta,\xi_1)  \int_{\theta}  e^{-2\pi i (x_2\cdot \xi_2)}\;d\lambda_{\theta}(x_2)\;d\gamma(\theta) \\
=&C\int_{\G}\hat{g}(\theta,\xi_1)  \delta_\theta(\xi_2)\;d\gamma(\theta) \\
=& C\int_{\{ \theta: \theta\perp\xi\}}  \widehat{g}(\theta,\xi)\; d\gamma_{\xi^\perp}(\theta)
\end{align*}
\end{proof} 

With these formulas in hand, we return to the relationship between $\Tnk$, $\Tnk^*$, and these mollified derivatives. 

\begin{proof}[Proof of Lemma \ref{commute}] 
We prove \eqref{Tsdls}. The proof of the other equation is similar.  
\[ \derivslx{s}{\Lambda} \Tnk^* g(x)=\left(  \widehat{\derivslx{s}{\Lambda} \Tnk^* g}(\xi) \right)^\vee =\left( \langle \xi \rangle ^s\langle\Lambda^{-1}\xi \rangle^{-s}\widehat{ \Tnk^* g}(\xi) \right)^\vee  \] 
\[ = \int_{\R^n}  e^{2\pi i x\cdot \xi }\langle \xi \rangle ^s\langle\Lambda^{-1}\xi \rangle^{-s} \widehat{ \Tnk^* g}(\xi)\; d\xi   . \] 
Applying Lemma \ref{DualF} to the right-hand side, 
\[ \derivslx{s}{\Lambda} \Tnk^* g(x) = \int_{\R^n} e^{2\pi i x\cdot \xi }\langle \xi \rangle ^s\langle\Lambda^{-1}\xi \rangle^{-s}\int_{\{ \theta: \theta\perp\xi\}}  \widehat{g}(\theta,\xi) \;d\gamma_{\xi^\perp}(\theta)\; d\xi. \] 
\[ = \int_{\R^n} e^{2\pi i x\cdot \xi }\int_{\{ \theta: \theta\perp\xi\}}  \langle \xi \rangle ^s\langle\Lambda^{-1}\xi \rangle^{-s}\widehat{g}(\theta,\xi) \;d\gamma_{\xi^\perp}(\theta)\; d\xi. \] 
By the definition of $\derivsly{s}{\Lambda}$, 
\[ = \int_{\R^n} e^{2\pi i x\cdot \xi }\int_{\{ \theta: \theta\perp\xi\}}\widehat{ \derivsly{s}{\Lambda}g}(\theta,\xi)\; d\gamma_{\xi^\perp}(\theta)\; d\xi. \] 
Using Lemma \ref{DualF}  again, 
\[ = \int_{\R^n} e^{2\pi i x\cdot \xi } \widehat{\Tnk^*[ \derivsly{s}{\Lambda}g]}(\xi) \; d\xi. \] 
 
\end{proof}

These mollified derivatives, like fractional derivatives, fail to satisfy Leibniz's rule for derivatives of products. However a modified Kato-Ponce inequality \cite{KP88} holds and is sufficient for our purposes: 

\begin{lem}\label{katoponce}
Let $d\geq2$. Let $s\in(0,\infty)$ and $\Lambda\in(4,\infty)$. Suppose that $r^{-1}=p_j^{-1}+q_j^{-1}$ for $j=1,2$ and that all exponents $r,p_j,q_j$ belong to the open interval $(1,\infty)$. Let $u\geq 0$ be a locally integrable function on $\R^n$. Suppose further that the weight $u$ belongs to $A_r$ and that $u=u_1v_1=u_2v_2$ where { $u_j^{p_j/r}\in A_{p_j}$ and $v_j^{q_j/r}\in A_{q_j}$.}   Then  there exists a constant $C$ such that for all $\Lambda>1$  the following inequality holds, whenever the right hand side is finite: 
\[ \|\derivslx{s}{\Lambda}(fg)\|_{ L^r(u)} \leq C \| \derivslx{s}{\Lambda}(f) \|_{L^{p_1}(u_1^{p_1/r})}\| g \|_{L^{q_1}(v_1^{q_1/r})}+ C\| f \|_{L^{p_2}(u_2^{p_2/r})}\| \derivslx{s}{\Lambda}(g) \|_{L^{q_2}(v_2^{q_2/r})}.  \] 
\end{lem} 

For standard fractional derivatives a similar weighted Lebesgue space result was proved in \cite{CUN16},\cite{CUN22err} for a larger range of exponents.  Our proof follows the methods of Christ and Weinstein in \cite{CW91}, but is rather long and consequently is deferred to the appendix. As we are specifically interested in the case where the weights are of the form $\langle x \rangle^\alpha$, it might be possible to prove the result for a wider range of weights and exponents as in \cite{OhWu22} but we do not pursue this here. 

We also require a version modified for the space $\M$.

\begin{lem}\label{mdgood}
 Let $s\in(0,\infty)$and $\Lambda\in(4,\infty)$. Suppose that $r^{-1}=p_j^{-1}+q_j^{-1}$ for $j=1,2$ and that all exponents $r,p_j,q_j$ belong to the open interval $(1,\infty)$. Let $u\geq 0$ be a locally integrable function on $\M$. Suppose the weight $u_\theta$ belongs to $A_r(\theta^\perp)$, uniformly in $\theta$, and that for each $\theta\in\G$ $u_\theta=u_{\theta,1}v_{\theta,1}=u_{\theta,2}v_{\theta,2}$ where $u_{\theta,j}^{p_j/r}\in A_{p_j}(\theta^\perp)$ and $v_{\theta,j}^{q_j/r}\in A_{q_j}(\theta^\perp)$.  Then there exists a constant $C$ such that $\derivsly{s}{\Lambda}(fg)\in L^r(\M,u)$ and the following inequality holds, whenever the right hand side is finite: 
\begin{multline*} 
\|\derivsly{s}{\Lambda}(fg)\|_{ L^r(\M,u)} \leq  C \| \derivsly{s}{\Lambda}(f) \|_{L^{p_1}(\M,u^{p_1/r})}\| g \|_{L^{q_1}(\M,v^{q_1/r})}+ \\ C\| f \|_{L^{p_2}(\M,u^{p_2/r})}\| \derivsly{s}{\Lambda}(g) \|_{L^{q_2}(\M,v^{q_2/r})}.
\end{multline*}
\end{lem}  
\begin{proof} By definition, 
\[ \|\derivsly{s}{\Lambda}(fg)\|_{ L^r(\M,u)}^r  = \int_{\G} \left( \int_{\theta^\perp} |\derivsly{s}{\Lambda}(fg)|^r u_\theta\; d\lambda_{\theta^\perp}y \right) \;d\theta.\] 
Applying Lemma \ref{katoponce} to the inner integral for each $\theta$ gives, 
\begin{multline*} 
\| \|\derivsly{s}{\Lambda}(fg)\|_{ L^r(\M,u)}^r \leq
 C \int_{\G} \left(  \| \derivslx{s}{\Lambda}(f) \|_{L^{p_1}(\theta^\perp, u_{\theta,1}^{p_1/r})}\| g \|_{L^{q_1}(\theta^\perp, v_{\theta,1}^{q_1/r})} \right. \\
\left. + \| f \|_{L^{p_2}(\theta^\perp, u_{\theta,2}^{p_2/r})}\| \derivslx{s}{\Lambda}(g) \|_{L^{q_2}(\theta^\perp, v_{\theta,2}^{q_2/r})} \right)^r\;d\theta.
\end{multline*}
Using Minowski's integral inequality, 
\begin{multline*} 
\| \|\derivsly{s}{\Lambda}(fg)\|_{ L^r(\M,u)}^r \leq
 C \int_{\G} \left(  \| \derivslx{s}{\Lambda}(f) \|_{L^{p_1}(\theta^\perp, u_{\theta,1}^{p_1/r})}\| g \|_{L^{q_1}(\theta^\perp, v_{\theta,1}^{q_1/r})}\right)^r\;d\theta \\
  +   C \int_{\G} \left(\| f \|_{L^{p_2}(\theta^\perp, u_{\theta,2}^{p_2/r})}\| \derivslx{s}{\Lambda}(g) \|_{L^{q_2}(\theta^\perp, v_{\theta,2}^{q_2/r})} \right)^r\;d\theta. 
\end{multline*}
The lemma then follows by H\"older's inequality. 
\end{proof}

These lemmas allow us analyze the smoothness of $\mathcal{S}f$ using its multilinear structure.

\begin{lem} \label{bigstepestimate} For all  $\rho \geq 0$ sufficiently small,  there exist $\varrho$ and $\varrho'$, such that $0<\varrho'\leq\varrho<\rho$ as well as a constant $C$ such that for all $s>0$ and $\Lambda\geq1 $, for all $f\in X_\rho$, 

\[ \| \derivslx{s}{\Lambda}\Sel f\|_{X_\varrho} \leq C \|f\|_{X_\rho}^{p_{el}q_{el}-1} \| \derivsly{s}{\Lambda} \Tnk f \|_{Y_{*,\varrho'}}.  
\] 
\end{lem}

\begin{proof}

The definition of $\Sel$ allows us to rewrite the left-hand side in such a way that our Kato-Ponce  Lemma \ref{katoponce} and Lemma \ref{mdgood}  together with the operator bounds from Lemma \ref{bddops} give the result for appropriately chosen $\varrho$ and $\varrho'$. The details follow below. 

Let $\rho>0$ be sufficiently small that $w^{\rho p_\rho}\in A_{p_\rho}$ (by Lemma \ref{ApWeights}) and take $f\in X_\rho$.  Choose $\varrho''$ and $\varrho$ such that $0<\varrho''\leq\varrho\leq \rho$ and sufficiently small that $w^{\varrho p'_{\varrho}/p_{el}}\in A_{p_\varrho}$ and $w^{\varrho'' p'_{\varrho''}/p_{el}}\in A_{p_{\varrho'}}$ (again by Lemma \ref{ApWeights} ), and   
\[p_{el}\varrho = \varrho'' + (p_{el}-1)\rho.\]

For this choice of $\varrho$ and $\varrho''$, using that $p_{el}=1/(p_0-1)$ and the definition of the $X_t$ and $Y_t$ spaces (\ref{def:XtYt}),  Lemma \ref{katoponce}  applies and gives that 
\begin{align*}
\| \derivslx{s}{\Lambda}\Sel f\|_{X_\varrho} = & \| \derivslx{s}{\Lambda} \left( \Tnk^* \left(\Tnk f\right)^{q_{el}}\right)^{p_{el}}\|_{X_\varrho}\\
\leq & C \sum_{i=1}^{p_{el}}\|   \Tnk^* \left(\Tnk f\right)^{q_{el}} \|_{Y_\rho}^{p_{el}-1} \left\| \derivslx{s}{\Lambda} \Tnk^* \left(\Tnk f\right)^{q_{el}} \right\|_{Y_{\varrho''}}   \\
  \leq & C  \|   \Tnk^* \left(\Tnk f\right)^{q_{el}} \|_{ Y_\rho}^{p_{el}-1}\left\| \Tnk^* \left(  \derivsly{s}{\Lambda} \left(\Tnk f\right)^{q_{el}} \right) \right\|_{Y_{\varrho''}}  
 \end{align*}
 where the second inequality follows from the observation that the summand is is independent of $i$ and Lemma \ref{commute}.  
 
Using the operator bounds from  Lemma \ref{bddops} and Kato-Pone inequality \ref{mdgood} and a similar argument to that above, with 
\( (q_0-1)\varrho''=\varrho'+(q_0-2)\rho\), yields

\begin{align*}
\| \derivslx{s}{\Lambda}\Sel f\|_{X_\varrho}  \leq & C  \left( \|  \left(\Tnk f\right)^{q_{el}} \|_{X_{*,\rho}}^{p_{el}-1} \right) \left\| \derivsly{s}{\Lambda} \left(\Tnk f\right)^{q_{el}} \right\|_{  X_{*,\varrho'}} \\
 \leq & C \left( \| f \|_{X_\rho}^{q_{el}(p_{el}-1) } \right) \sum_{i=1}^{q_{el}} \left\| \derivsly{s}{\Lambda} \Tnk f \right\|_{Y_{*,\varrho'}}\left\|  \Tnk f\right\|_{Y_{*,\rho}}^{q_{el}-1}  \\
 \leq & C \left( \| f \|_{X_\rho}^{q_{el}p_{el}-1 } \right) \left\| \derivsly{s}{\Lambda} \Tnk f \right\|_{Y_{*,\varrho'}} 
\end{align*}

\end{proof}

We require one last technical lemma: 
 
\begin{lem} \label{niceD1} For all $s,\gamma\in\R$ such that  $s\geq\gamma\geq0$,  for all sufficiently small $t\geq 0$, there exists a constant $C$ such that for all  $\Lambda>1$, for all $h\in X_t$, 
\[  \|  \derivslx{s}{\Lambda}\derivsx{-\gamma} f \|_{X_t}\leq C\|  \derivslx{s}{\Lambda} f \|_{X_t}^{1-\gamma/s} \|f\|_{X_t}^{\gamma/s}. \] 
\end{lem}
\begin{proof}
First, the operator $ \derivslx{s}{\Lambda}\derivsx{-s} $ is a Fourier multiplier operator given by the multiplier  
\[ m(\xi)= \langle \xi\rangle^{-s}\frac{\langle \xi \rangle ^s}{\langle\Lambda^{-1}\xi \rangle ^s}=  \frac{1}{\langle\Lambda^{-1}\xi \rangle ^s}. \] 
For each multi-index $\alpha$, there exists $C_\alpha$ independent of $\Lambda\geq 1$  such that
\[ \left| \partial_{\xi}^\alpha(m)\right|\leq  C_{\alpha} |\xi|^{-|\alpha|}.\]
Therefore (\cite{Stein}, pg 26,205), 
\[  \| \derivslx{s}{\Lambda}\derivsx{-s} f\|_{X_t} \leq C \| f \|_{X_t}. \] 
 Additionally, $\derivsx{-i\sigma}$ is bounded on $X_t$ with a norm $\lesssim \langle \sigma \rangle ^c$ uniformly for all $\sigma$ \cite{CX}. Whence, 
\[  \| \derivslx{s}{\Lambda}\derivsx{-s+i\sigma} f\|_{X_t} \leq C \langle \sigma \rangle ^c \| f \|_{X_t}. \] 
Trivially, $\|\derivslx{s}{\Lambda}f\|_{X_t} \leq C \|\derivslx{s}{\Lambda}f\|_{X_t}$. 

The lemma follows from these two estimates by complex interpolation applied to the the analytic family of operators $ \derivslx{s}{\Lambda}\derivsx{-z}$.    
\end{proof}

\section {Conclusion of Proof}\label{sec:conclusion} 
This section is devoted to the proof of the following lemma from which Theorem \ref{main} follows.  

\begin{prop}\label{allthederivatives} Let $n\geq 2$,  $q_0\in(1,n+1]$, such that for $p_0=\frac{ nq_0}{n-k+kq_0}$, $q_{el}=q_0-1$ and $p_{el}=\frac{1}{p_0-1}$ are both integers, and fix $\lambda\in\R$. Take $\rho>0$ and let $f\in X_\rho$ be any real-valued solution of the generalized Euler-Lagrange equation $f=\lambda \Sel f.$ Then there exists $\varrho\in(0,\rho)$,  such that for all $s\geq 0$, $\derivsx{s} f \in X_\varrho$.  
\end{prop} 

\begin{proof}[Proof of Theorem \ref{main} using Proposition \ref{allthederivatives}] 
By the decay estimate of Proposition \ref{extraDecay}, if $f\in L^{p_0}(\R^n)$ is a solution of the generalized Euler-Lagrange equation $f=\lambda \Sel f$, there exists $t>0$ such that $f\in X_t$. Thus, the conditions of Proposition \ref{allthederivatives} are met, and there exists $\varrho>0$ such that for all $s\geq 0$, $\derivsx{s} f \in X_\varrho$.  The theorem then follows by Sobolev embedding (see for instance, \cite{Stein70}).

\end{proof}

\begin{proof}[Proof of Proposition \ref{allthederivatives} ]

Fix $\lambda\in\R$.  Let $f\in X_\rho$ for some $\rho>0$ be any solution of the generalized Euler-Lagrange equation $f=\lambda \Sel f $. Let $\gamma(\rho)$ be the small parameter from Lemma \ref{smoothing} which applies because if $p_{el}=\frac{1}{p_0-1}$ is an integer then $p_0\in(1,2]$.

Fix $s>\gamma(\rho)$. Note that if $\derivsx{s} f \in X_\varrho$ for all such $s>\gamma(\rho)$, then interpolation and $ f \in X_\rho \subset X_\varrho$ guarantees the result for all $s$. 

It suffices to consider $\|f\|_{X_\rho}= 1$, as $F= f/\|f\|_{X_\rho}$ will satisfy $f=\lambda\|f\|_{X_\rho}^{p_{el}q_{el}}  \Sel f $. Further, it is enough to prove that there exists a finite constant $C$ independent of $\Lambda$ such that for all $\Lambda\geq 4$, $\|\derivslx{s}{\Lambda}f \|_{X_\varrho} <C$. Then taking $\Lambda\to\infty$, we have $\|\derivsx{s} f \|_{X_\varrho} < C$ for all $s>\gamma(\rho)$.

Using Corollary \ref{bigstepestimate} 
there exist $\varrho'<\varrho$ both in $(0,\rho)$,
\begin{equation}\label{bigstep} \| \derivslx{s}{\Lambda}f\|_{X_\varrho} \leq C \| f\|_{X_\rho}^{p_{el}q_{el}-1}\|\Tnk  \derivslx{s}{\Lambda}f \|_{Y_{*,\varrho'}}= C \|  \Tnk \derivslx{s}{\Lambda} f \|_{Y_{*,\varrho'}}.
\end{equation}
Then by convexity of the $Y_{*,t}$ norms (Lemma \ref{littlethings}), there exists $\theta\in(0,1)$ such that  
\[ \|\Tnk \derivslx{s}{\Lambda} f \|_{Y_{*,\varrho'}}\leq \|   \Tnk \derivslx{s}{\Lambda} f \|_{Y_{*,\varrho}}^\theta\|    \Tnk \derivslx{s}{\Lambda} f \|_{Y_{*,0}}^{1-\theta}. \]
Applying the operator bounds $\Tnk$ (Lemma \ref{bddops}), 
\[ \|   \Tnk \derivslx{s}{\Lambda} f \|_{Y_{*,\varrho'}}\leq C \| \derivslx{s}{\Lambda} f \|_{X_{\varrho}}^\theta\|    \Tnk \derivslx{s}{\Lambda}  f \|_{Y_{*,0}}^{1-\theta}. \]
From this and \eqref{bigstep}, we conclude that 
\begin{equation}\label{bddbylq} 
\| \derivslx{s}{\Lambda} f\|_{X_\varrho}\leq C \| \Tnk  \derivslx{s}{\Lambda} f \|_{L^{q_0}(\mathcal{M}_{n,k})}. 
\end{equation}  

Choose $\gamma=\gamma(\varrho)$ as guaranteed by Lemma \ref{smoothing}. Then using properties of the $k$-plane transform \eqref{commute}, 
$\Tnk \derivslx{s}{\Lambda} f= \derivsy{\gamma} \Tnk(  \derivslx{s}{\Lambda}  \derivsx{-\gamma} f)$.  We now apply Lemma \ref{smoothing}, 
\[ \| \Tnk \derivslx{s}{\Lambda}  f \|_{L^{q_0}(\M)} = \| \derivsy{\gamma} \Tnk( \derivslx{s}{\Lambda}  \derivsx{-\gamma} f)\|_{L^{q_0}(\M)}\leq C  \|  \derivslx{s}{\Lambda}  \derivsx{-\gamma} f \|_{X_\varrho}. \]

Applying Lemma \ref{niceD1}, $\|f\|_{X_\rho}= 1$, and that the $X_t$ spaces are nested,
\[  \|  \derivslx{s}{\Lambda}  \derivsx{-\gamma} f \|_{X_\varrho} \leq C  \|  \derivslx{s}{\Lambda}   f \|_{X_\varrho}^{1-\gamma/s} \| f\|_{X_{\varrho}}^{\gamma/s} \leq C  \|  \derivslx{s}{\Lambda}   f \|_{X_\varrho}^{1-\gamma/s}. \]
Combining this estimate and \eqref{bddbylq}, for $C$ independent of $\Lambda$,
\[  \|  \derivslx{s}{\Lambda}  f \|_{X_\varrho} \leq C  \|  \derivslx{s}{\Lambda}  f \|_{X_\varrho}^{1-\gamma/s}.\] 
If $ \|  \derivslx{s}{\Lambda}  f \|_{X_\varrho}$, is finite, we may then divide through and conclude uniform bounds. This finitenss is a consequence of Lemma \ref{niceMD}, which says that $\|  \derivslx{s}{\Lambda} f \|_{X_\varrho}\leq C_{\Lambda} \|f \|_{X_\varrho}$. Therefore, 
\[ \|  \derivslx{s}{\Lambda}  f \|_{X_\varrho}\leq C \] 

where $C$ is independent of $\Lambda$. 
\end{proof}

\appendix

\section{Existence of extremizers when $q$ is an integer}

Existence of extremizers for $q=n+1$ is known, Drouot \cite{D11}, using work of Lieb \cite{HLS}. Our proof follows these methods. The idea is simple: show some extremizing sequence converges. In particular, we show that any radial and symmetrically decreasing extremizing sequence is pre-compact, which is enough as such extremizing sequences exist when $q$ is an integer.

The key tool in reducing to the radial semetrically decreasing case is the following rearrangement theorem.  Christ proves the result for $q=n+1$ in \cite{C84}, in a multilinear form from which the inequality for smaller $q$ can be derived.  Another proof is given in \cite{BL97}.   It is an open question whether the rearrangement inequality holds in the non-integer case. If it does, extremizers for \eqref{LpLq} exist for all $q\in(1,n+1]$ using this same proof. 

\begin{lem}[$k$-plane rearrangement, \cite{C84}]\label{kprearrangement} When $q\in[1,n+1]$ is an integer, for all $f\in L^p(\R^n)$, 

\begin{equation*}
		\|\Tnk f \|_{L^q} \leq \| \Tnk f^* \|_{L^q}. 
	\end{equation*}
	
\end{lem}

As $\|f\|_{L^p}= \|f^*\|_{L^p}$,  Lemma \ref{kprearrangement} gives that if $\{f_n\}$ is any sequence along which \eqref{func} converges to its supremum then $\{f_n^*\}$ is as well. Therefore to find sharp constant $A_0$ in \eqref{LpLq} it's is enough to consider radial functions, i.e.
\begin{equation}\label{radsup} 
A_0 = \sup_{\{f: \|f\|_{L^p}=1,f=f^*\}}\frac{\| \Tnk f \|_{L^q} }{\|f\|_{L^p}}.
\end{equation}

With this notation our main result is:

\begin{thm}\label{thm:integerexistence} Let $q\in(1,n+1]$ be an integer. Let $f_j$ be an extremizing sequence for \eqref{radsup}.  Then there exist a subsequence, which we will continue to denote $f_j$, a  sequence $\sigma_j\in(0,\infty)$, and a non-zero function $f\in \lp$ such that the new extremizing sequence $|\sigma_j|^{-n/p}f_j(\sigma_j^{-1} x)\to f$ in $L^p(\R^n)$. In particular, the supremum in \eqref{radsup} is attained. \end{thm} 

This result was proved by Drouot in \cite{D11} for the $q=n+1$ case. A more general version of this result has been proved by Christ in \cite{C11} in the case that $q=n+1$ and $k=n-1$.  

Restricting to sequences of radial, symmetricly decreasing functions allows us to extract a sequence which converges pointwise (except perhaps at zero) by Helly's selection principle. 

\begin{lem}[Helly's selection principle (see for instance \cite{LLtext})]\label{Helly} 
For any sequence, $\{f_j\}$, of functions $[0,\infty)\to[0,\infty)$ (resp. $\R^n\to[0,\infty)$) which are decreasing (resp. radial symmetric decreasing), and for which there exists a finite positive constant $B$, such that $\|f_j\|_{L^p}\leq B$ for all $j$, there exists a function $f\in L^p$ and a subsequence, which we will continue to denote $f_j$, such that $f_j\to f$ pointwise except perhaps at the origin.
\end{lem} 

Given pointwise convergence, we now wish to apply a powerful theorem of Lieb's from \cite{HLS}. 

\begin{lem}[Lieb \cite{HLS} Lemma 2.7]\label{pointwiseisenough} Let $(M,\Sigma,\mu)$ and $(M',\Sigma',\mu')$ be measure spaces and let $X$ (resp. $Y$) be $L^p(M,\Sigma,\mu)$ (resp. $L^q(M',\Sigma',\mu')$) with $1\leq p\leq q<\infty$. Let $A$ be a bounded linear operator from $X$ to $Y$. For $f\in X$, $f\neq 0$, let 
\[ R(f)= \|Af\|_Y/\|f\|_X \text{  and   } N=sup\{R(f)|f\neq 0\}\] 
Let $\{f_j\}$ be a uniformly norm-bounded maximizing sequence for $N$ and suppose that $f_j\to  f\neq0$ and that $Af_j\to Af$ pointwise almost everywhere. Then $f$ maximizes $R$, ie, $R(f)=N$. Moreover, if $p<q$ and if $lim\|f_j\|_X=C$ exists, then $\|f\|_X=C$ and hence $\|Af_j\|_Y\to\|Af\|_Y$. 
\end{lem} 

Thus we are required to prove that  there exists a sequence $\sigma_j\in(0,\infty)$ such that the new extremizing sequence $|\sigma_j|^{-n/p}f_j(\sigma_j^{-1} x)$ converges pointwise to a nonzero function and that $\Tnk f_j\to \Tnk f$ pointwise. 

\begin{lem}\label{gooddilations}  Let $f_j$ be an extremizing sequence for \eqref{radsup} such that for each $j$, $\|\Tnk f_j\|_{L^q}\geq A_0/2$. Then there exists a sequence $\sigma_j\in(0,\infty)$ and a ball $\mathcal{B}$ such that for each $j$, $|\sigma_j|^{-n/p}f_j(\sigma_j^{-1} x)\geq \one_{\mathcal{B}}(x)$. 
\end{lem} 

This follows from the known Lorentz norm estimates for the $k$-plane transform. 

\begin{lem}\label{Clorentz} \cite{C84} There exists a constant $C$ such that for all $f\in L^{\frac{n+1}{k+1},n+1}$, 
\[ \|\Tnk f\|_{L^{n+1}(\M)} \leq C \|f\|_{L^{\frac{n+1}{k+1},n+1}}.\]
\end{lem} 
By interpolation, this yields 
\begin{lem}\label{lorentz}  For $q\in[1,n+1]$ and $p=\frac{nq}{n-k+kq}$, there exists a constant $C$ such that for all $f\in L^{p,q}$, 
\[ \|\Tnk f\|_{L^{q}(\M)} \leq C \|f\|_{L^{p,q}}.\]
\end{lem} 
\begin{proof} 
The $q=n+1$ endpoint is exactly Christ's estimate. The $q=1$ endpoint is the inequality $\|\Tnk f\|_{L^{1}(\M)} \leq \|f\|_{L^{1}}$. These estimates certainly imply the corresponding estimates for weak $L^q$.  Thus, using the off-diagonal Marcinkiewicz interpolation theorem \cite{SW59},  for $q\in(1,n+1)$ and $p=\frac{nq}{n-k+kq}$, for all $r\in(0,\infty]$, there exists a constant $C$ such that for all $f\in L^{p,r}$, 
\[ \|\Tnk f\|_{L^{q,r}(\M)} \leq C \|f\|_{L^{p,r}}.\]
Taking $r=q$ gives the theorem. 
\end{proof}

This proof is essentially the same as that given in \cite{D11}.  

\begin{proof}[Proof of Lemma \ref{gooddilations}] 
First note that for $1\leq p<q<\infty$,
\begin{equation}\label{weakbd} 
\|f\|_{L^{p,q}}^q\leq \|f\|_{L^{p,\infty}}^{q-p}\|f\|_p^p
\end{equation} 
which can be seen as follows. 
\begin{align*} \|f\|_{L^{p,q}}^q = &\int_0^\infty t^q \left(\mu\{x:|f(x)|>t\}\right)^{q/p} \;\frac{dt}{t}\\
=& \int_0^\infty \left(t^{p} \mu\{x:|f(x)|>t\}\right)(t^{q-p} \left(\mu\{x:|f(x)|>t\}\right)^{(q-p)/p} 
\frac{dt}{t} \\
\leq& \sup_{t>0}(t^{q-p} \left(\mu\{x:|f(x)|>t\}\right)^{(q-p)/p} \int_0^\infty \left(t^{p} \mu\{x:|f(x)|>t\}\right) 
\frac{dt}{t} \\
\leq&\| f\|_{L^{p,\infty}}^{q-p}\|f\|_{L^p}^p. 
\end{align*}

Therefore as for each $j$, $\|\Tnk f_j\|_{L^q}\geq A_0/2$ and $\|f_j\|_{L^p}=1$, 
\begin{align*} (A_0/2)^q \leq & \|\Tnk f_j\|_{L^q}^q\\
\leq & C\|f_j\|_{L^{p,q}}^q\\
\leq & C\|f_j\|_{L^{p,\infty}}^{q-p}\\
\leq& C \sup_{t>0} (t^{q-p} \left(\mu\{x:|f_j(x)|>t\}\right)^{(q-p)/p}
\end{align*}
Thus, there exists $t_j$ such that 
\[  t_j^{p} \left(\mu\{x:|f_j(x)|>t_j\}\right)\geq \left( \frac{A_0^q}{2^{q+1}C}\right)^{\frac{p}{q-p}}.
\] 
Set $\sigma_j=t_j^{-p/n}$.  Then 
\[ \mu\{x:\sigma_j^{n/p}|f_j(\sigma_j x)|>1\}=
 t_j^{p} \left(\mu\{x:|f_j(x)|>t_j\}\right) .\]
Let $\mathcal{B}$ be the ball centered at origin with volume $ \left( \frac{A_0^q}{2^{q+1}C}\right)^{\frac{p}{q-p}}$. As, $f_j$ is symmetric decreasing, it follow from the above estimates that 
\[ \sigma_j^{n/p}f_j(\sigma_j x)\geq \one_{\mathcal{B}}(x). 
\qedhere \] 
\end{proof} 

\begin{proof}[Proof of Theorem \ref{thm:integerexistence} ]
 Let $f_j$ be an extremizing sequence for \eqref{radsup}.  Thus $\|f_j\|_{L^p}=1$, $f_j$ is radial symmetric decreasing, and there exists a subsequence, which we will continue to denote $f_j$, such that $\|\Tnk f\|_{L^q} \geq A_0/2$. Take $\sigma_j$ as guaranteed by Lemma \ref{gooddilations}.  Applying Helly's selection principle (Lemma \ref{Helly})  to $|\sigma_j|^{-n/p}f_j(\sigma_j^{-1} x)$, there exist a function $f\in L^p$ and a subsequence, which we will continue to denote $f_j$, such that $|\sigma_j|^{-n/p}f_j(\sigma_j^{-1} x)\to f$ pointwise except at the origin. By Lemma \ref{gooddilations} for each $j$,   $|\sigma_j|^{-n/p}f_j(\sigma_j^{-1} x)\geq \one_{\mathcal{B}}$. Thus $f\geq \one_{\mathcal{B}}$, and in particular, $f\neq 0$. 

As $f_j$ is radial symmetric decreasing, $\Tnk f_j(\theta,y)$ is as well. Let $g_j=\Tnk f_j(\theta,y)$. Applying Helly's selection principle there exists $g$ such that some subsequence still denoted $g_j$ satisfies $g_j\to g$ pointwise except at the origin. As $q\in(1,n+1]$, $p\in(1,\infty)$. Also $\|f_j\|_{L^p(\R^n)}=1$, and so by \eqref{LpLq}, $\|g_j\|_{\lqm}=\|\Tnk f_j(\theta,y)\|_{\lqm}\leq A_0$. Thus each of these sequences converges weakly. Finally, $T$ is a bounded linear operator for $L^p$ to $\lqm$ and thus also is a bounded linear operator from $L^p(\R^n)$ to $\lqm$ when each is endowed with the weak topology. Hence $g=\Tnk f$, and the conditions of Lemma \ref{pointwiseisenough} are satisfied. 

\end{proof}

\section{Proof of the fractional derivative inequality for mollified derivatives}\label{sec:kp}
This proof generally follows the argument for standard fractional derivatives given by Christ and Weinstein in \cite{CW91} and in \cite{CX} for mollified derivatives, however we work with a different definition of the mollified derivatives correcting a technical issue in \cite{CX}, where the multipliers have the same asymptotic behavior but are not differentiable. The key ideas are a Littlewood-Paley decomposition and weighted Calderon-Zygmund theory, but we go through the details slowly because independence of the constants on $\Lambda$ is critical for our application. To show that various operators appearing in the course of proof are Calderon-Zygmund operators, we'll need the following estimates: 

\begin{lem} For each multi-index $\alpha$, such that $|\alpha|\leq\frac{n}{2}+1$, the multiplier $m_{\Lambda,s}(\xi)$ given by 
\[   m_{\Lambda,s}(\xi)= \frac{\langle \xi \rangle ^s}{\langle\Lambda^{-1}\xi \rangle ^s}\]
enjoys the decay estimates
\begin{align}\label{eq:LambdaDecay}  \left| \partial^\alpha_\xi \left(\frac{\langle \xi\rangle^s}{\langle\Lambda^{-1} \xi\rangle^s} \right) \right|
&\leq \frac{C_{ \alpha}\Lambda^{s}}{|\xi|^{|\alpha|}}\\
\label{eq:DecayNoLambda}  
 \left| \partial^\alpha_\xi \left(\frac{\langle \xi\rangle^s}{\langle\Lambda^{-1} \xi\rangle^s} \right) \right|
&\leq{C_\alpha}\langle \xi\rangle^{s-|\alpha|}.
\end{align}
\end{lem}
\begin{proof} The proof is essentially a direct computation. Let $f,g:\R\to\R$  be the functions, $f(x)=(x)^{s/2}$ and $g(x)=\frac{1+x}{\Lambda^2+x}$, so that $ m_{\Lambda,s}(\xi)=\Lambda^sf(g(|\xi|^2))$. Using $f^{(\beta)}$ to denote the $\beta$-th derivative, for each $\beta,$
\[ | f^{(\beta)}(g(|\xi|^2))|=C_{s,\beta}\left|\frac{\langle \xi \rangle^{s-2\beta}}{(\Lambda^2+\xi^2 )^{s/2-\beta} }\right| \lesssim \min(1, \Lambda^{-s+2\beta}\langle\xi\rangle^{s-2\beta}) \]
and 
\[ |g^{(\beta)}(|\xi|^2)|=C_{s,\beta}\left|\frac{\Lambda^2-1}{(\Lambda^2+|\xi|^2)^{\beta+1}}\right| \lesssim \min\left( \frac{1}{|\xi|^{2\beta }}, \frac{1}{\Lambda^{2\beta}},\frac{1}{(\Lambda|\xi|)^{\beta}}\ldots\right).\]
These estimates together with the multivariable chain rule complete the proof. To avoid messy notation, we give the details for only for $\partial^\alpha=\frac{\partial^2}{\partial_i^2}.$
Expanding, 
\[ \frac{\partial^2}{\partial_i^2}f(g(|\xi|^2))= f''(g(|\xi|^2))(g'(|\xi|^2)2\xi_i)^2+ f'(g(|\xi|^2))g''(|\xi|^2)(2\xi_i)^2 +2f'(g(|\xi|^2))g'(|\xi|^2)\] 
Using the first bounds above, then, $\left| \frac{\partial^2}{\partial_i^2} \left(\frac{\langle \xi\rangle^s}{\langle\Lambda^{-1} \xi\rangle^s} \right) \right|\lesssim\Lambda^s \frac{\xi_i^2+1}{|\xi|^4} $ which suffices to show \eqref{eq:LambdaDecay} 
Alternately, using the second bound for $f^{(\beta)}$ and either the second or third bound for $g^{(\beta)}$, yields: 
\[ \left| \frac{\partial^2}{\partial_i^2}\left(\frac{\langle \xi\rangle^s}{\langle\Lambda^{-1} \xi\rangle^s} \right) \right|\lesssim \langle \xi\rangle^{s}\left( \frac{\Lambda^4}{\langle \xi\rangle^{4}}\cdot\frac{\xi_i^2}{\Lambda^4}+   \frac{\Lambda^2}{\langle \xi\rangle^{2}}\cdot\frac{\xi_i^2}{\Lambda^2|\xi|^2}+ \frac{\Lambda^2}{\langle \xi\rangle^{2}}\cdot\frac{1}{\Lambda^2}\right) \]
which suffices to show \eqref{eq:DecayNoLambda}
\end{proof}

\begin{lem} For each multi-index $\alpha$, such that $|\alpha|\leq\frac{n}{2}+1$, the multiplier $m_{\Lambda,s}(\xi)$ given by 
\[   m_{\Lambda,s}^{-1}(\xi)= \frac{\langle\Lambda^{-1}\xi \rangle ^s}{\langle \xi \rangle ^s}\]
enjoys the decay estimate
\begin{align}\label{eq:inverseDecay}  \left| \partial^\alpha_\xi \left(\frac{\langle\Lambda^{-1} \xi\rangle^s}{\langle \xi\rangle^s} \right) \right|
&\leq  {C_{ \alpha} }{\langle\xi\rangle^{-s-|\alpha|}}
\end{align}
\end{lem}
\begin{proof} Again, the proof is essentially a direct computation. Let $f,g:\R\to\R$  be the functions, $f(x)=(x)^{s/2}$ and $\tilde{g}(x)=\frac{\Lambda^2+x}{1+x}$, so that $ m_{\Lambda,s}^{-1}(\xi)=\Lambda^{-s}f(\tilde{g}(|\xi|^2))$. Using $f^{(\beta)}$ to denote the $\beta$-th derivative, for each $\beta,$
\[ |f^{(\beta)}(\tilde{g}(|\xi|^2))|=\left|C_{s,\beta}\frac{(\Lambda^2+\xi^2 )^{s/2-\beta} }{\langle \xi \rangle^{s-2\beta}} \right| \lesssim  \Lambda^{s-2\beta}\langle\xi\rangle^{-s+2\beta} \]
and 
\[ |\tilde{g}^{(\beta)}(|\xi|^2)|= C_{s,\beta}\left|\frac{1-\Lambda^2}{(1+|\xi|^2)^{\beta+1}} \right| \lesssim  \Lambda^2\langle\xi\rangle^{-2(\beta+1)}.\]
These estimates together with the multivariable chain rule complete the proof. 
\end{proof}

\begin{proof}[Proof of Lemma \ref{katoponce}] 
Fix $\eta\in\mathscr{S}(\R^n)$, a radial function, such that $\eta(\xi)=1$ if $|\xi|\leq 1$ and $\eta(\xi)=0$ if $|\xi|\geq 2$. For each $j\in\{0,1,2,\ldots\}$, define $P_j$ by $\widehat{P_jf}(\xi)=\hat{f}(\xi)\eta(2^{-j}\xi)$. For $j\geq 1$, define $Q_j$ by $Q_j=P_j-P_{j-1}$.

  Fix $\kappa$ such that $2^{\kappa-1}<  \Lambda\leq 2^{\kappa} $.  Define $R_\kappa f = f - P_{\kappa}f$.
Note that $\widehat{P_{j}f}$ is supported in $\{\xi:|\xi|\leq2^{j+1}\}$, $\widehat{Q_{j}f}$ is supported in $\{\xi:2^{j-1}\leq|\xi|\leq2^{j+1}\}$, and  $\widehat{R_\kappa f}(\xi)$ is supported in $\{\xi:|\xi|\geq 2^\kappa\}$.

For any $j_0\in [1,\kappa-1]$ decompose, 
\[ f = P_\kappa f+R_\kappa f = P_{j_0-1}f+ \sum_{j=j_0}^\kappa Q_{j} f + R_\kappa \] 
and decompose $g$ similarly. Using such decompositions and re-organzing terms,
\begin{eqnarray} 
 \label{eq:matched}  fg & = &  \sum_{j=1}^\kappa Q_jf\cdot P_{j} g +  \sum_{j=2}^\kappa Q_jg\cdot P_{j-1} f \\
\label{eq:remainder}	&&+ R_\kappa f P_\kappa g +  R_\kappa g  P_{\kappa}f+ R_\kappa f \cdot R_\kappa g \\
 \label{eq:small} & & + P_0 f  P_0 g 
\end{eqnarray} 

While the proof follows the outline from \cite{CW91,CX}, this decomposition groups more terms together.  The argument for each term in \eqref{eq:matched} follows the argument for terms supported on a ball from \cite{CW91,CX}.

\subsection{Remainder Terms}
We start with the terms in \eqref{eq:remainder}. Consider $\|\derivslx{s}{\Lambda}( R_\kappa f\cdot P_{\kappa} g)\|_{L^r(u)}$ .  Write, 
\[ \widehat{\derivslx{s}{\Lambda}(R_\kappa f P_\kappa g)}= \frac{2^{-s\kappa}\langle \xi \rangle ^s}{\langle{\Lambda^{-1}\xi}\rangle^s}\left(2^{s\kappa} \widehat{R_\kappa f}*\widehat{P_\kappa g}\right). \] 
 Recall that $ \Lambda\leq 2^\kappa$. Making use of the estimate \eqref{eq:LambdaDecay}, 
\[   \left| \partial^\alpha_\xi \left(\frac{2^{-s\kappa}\langle \xi\rangle^s}{\langle\Lambda^{-1} \xi\rangle^s} \right) \right| 
\leq \frac{C_{\alpha}\Lambda^{s}2^{-s\kappa}}{|\xi|^{|\alpha|}}\leq \frac{C_{\alpha}}{|\xi|^{|\alpha|}}.
\] 
Therefore, as $u\in A_r$ by the H\"ormander-Mihlin multiplier theorem (see e.g \cite{Stein} pg 26, 205),
\[ \|\derivslx{s}{\Lambda}( R_\kappa f\cdot P_\kappa g)\|_{L^r(u)} \lesssim \|2^{s\kappa} R_\kappa f\cdot P_\kappa g\|_{L^r(u)}.\] 
Using that $u=u_1v_1$ and H\"older's inequality, 
\[ \|2^{s\kappa} R_\kappa f\cdot P_\kappa g\|_{L^r(u)}\lesssim \left\|2^{s\kappa} R_\kappa  f  \right\|_{L^{p_1}(u_1^{p_1/r})} \left\| P_\kappa g\right\|_{L^{q_1}(v_1^{q_1/r})}  .\]

Letting $\mathcal{M}$ denote the Hardy-Littlewood Maximal function, by definition of $P_\kappa$ and using that  $v_1^{q_1/r}\in A_{q_j}$, 
\[  \left\| P_\kappa g\right\|_{L^{q_1}(v_1^{q_1/r})}\lesssim \left\| \mathcal{M}g\right\|_{L^{q_1}(v_1^{q_1/r})}\lesssim \left\|g\right\|_{L^{q_1}(v_1^{q_1/r})}.\]
Next for $ \left\| 2^{s\kappa}R_\kappa f\right\|_{L^{p_1}(u_1^{p_1/r})}$, write
\[  \widehat{2^{s\kappa}R_\kappa f}=\left(\frac{2^{s\kappa}(1-\eta(2^{-\kappa}))\langle{\Lambda^{-1}\xi}\rangle^s}{\langle \xi \rangle ^s}\right) \frac{\langle \xi \rangle ^s}{\langle{\Lambda^{-1}\xi}\rangle^s}\hat{f}(\xi) .  \]
Using \eqref{eq:inverseDecay}, 
\[ \left| \partial^\alpha_\xi \left( \frac{2^{s\kappa}\langle{\Lambda^{-1}\xi}\rangle^s}{\langle \xi \rangle ^s}\right)  \right| 
\leq C_{\alpha}2^{s\kappa}   {\langle\xi\rangle^{-s-|\alpha|}}. \] 
Now,  $2^{\kappa}\leq 2\Lambda$ and $\eta$ is Schwartz and supported on $|\xi|\leq 2$. Hence, when the derivative is nonzero, $|\xi|>\Lambda\geq 2^{\kappa-1}$, and 
\[ \left| \partial^\alpha_\xi \left( \frac{2^{s\kappa}(1-\eta(2^{-\kappa}))\langle{\Lambda^{-1}\xi}\rangle^s}{\langle \xi \rangle ^s}\right)  \right| 
 \leq C_{\alpha}\langle\xi\rangle^{-|\alpha|}. \] 
Again by the by the H\"ormander-Mihlin multiplier theorem, 

\[  \left\|2^{s\kappa}R_\kappa f\right\|_{L^{p_1}(u_1^{p_1/r})}\lesssim \left\| \derivslx{s}{\Lambda}f\right\|_{L^{p_1}(u_1^{p_1/r})}.\]

The other remainder terms are addressed analogously with only the additional observation that $ \left\| R_\kappa g\right\|_{L^{q_1}(v_1^{q_1/r})}= \left\|g-P_\kappa(g)\right\|_{L^{q_1}(v_1^{q_1/r})}$ and thus,
\[ \left\| R_\kappa g\right\|_{L^{q_1}(v_1^{q_1/r})}\leq C \left( \left\|  g\right\|_{L^{q_1}(v_1^{q_1/r})}+ \left\| \mathcal{M} g\right\|_{L^{q_1}(v_1^{q_1/r})}\right) \leq C \left\|  g\right\|_{L^{q_1}(v_1^{q_1/r})}.\] 

The small term in \eqref{eq:small} is handled similarly, noting that the restriction on the support of $P_0$ is easily sufficient to prevent $\Lambda$ dependence. 

\subsection{Main Terms}
To estimate $\|\derivslx{s}{\Lambda}(fg)\|_{ L^r(u)}$, next consider the case where the derivative falls on first sum in line $\eqref{eq:matched}$. For each $j\in[1,\kappa]$, $\derivslx{s}{\Lambda} Q_jf P_{j} g$ has Fourier support in the ball $|\xi|<2^{j+2}$. The analysis for these terms follows similarly to that in \cite{CW91} with the addition observations that our weights are in the appropriate $A_p$ classes and our mollified derivatives satisfy the decay estimates above.

 Invoking weighted Littlewood-Paley theory\footnote{ the result used here follows from results in \cite{Stein}, pg 267,205 see also \cite{KurtzWeightedLP,Rychkov01LP}}  as $u\in A_r$, 
\begin{align*},  
\left\| \derivslx{s}{\Lambda} 
\left(\sum_{j=1}^\kappa Q_jf\cdot P_{j} g\right) \right\|_{L^r(u)} 
 \simeq &  \left\| \left\{  Q_l \derivslx{s}{\Lambda} \left( \sum_{j=1}^\kappa  Q_jf P_{j} g\right) \right\}_{ l=-\infty}^\infty\right\|_{L^r(\ell^2,u)}. 
\end{align*}

 The first step is to prove that the mollified derivative acts like multiplication by $2^{sl}$ on each annuli independent of $\Lambda.$ For any function $h_l$ supported in $\{\xi:|\xi|\leq 2^{l+1}\}$, 
\[ \widehat{\derivslx{s}{\Lambda}( h_l) } = \frac{2^{-ls}\langle \xi\rangle ^s }{\langle\Lambda^{-1} \xi\rangle ^s} \eta(2^{-l-2}\xi)2^{ls}\widehat{h_l}(\xi) . \]
Define $M_l(h)$ by $\widehat{M_l h}=m_l \widehat{h}$ where $m_l= \frac{2^{-ls}\langle \xi\rangle ^s }{\langle\Lambda^{-1} \xi\rangle ^s} \eta(2^{-l-2}\xi)$. Let $\vec{M}$ be the vector valued operator $\vec{M}(\{ h_l \})= \{M_lh_l\}$.  Thus, 
\[   \left\| \left\{Q_l \derivslx{s}{\Lambda} \left(\sum_{j=3}^\infty Q_jf P_j g\right) \right\}_{ l=-\infty}^\infty \right\|_{L^r(\ell^2,u)}= \left\|\vec{M} \left\{ 2^{ls}  Q_l \left(\sum_{j=3}^\infty Q_jf P_j g\right) \right\}_{ l=-\infty}^\infty\right\|_{L^r(\ell^2,u)}. \] 
Using \eqref{eq:LambdaDecay} and that $\eta$ is a Schwartz function supported on $|\xi|\leq 2$ so that anywhere the derivative is nonzero $|\xi|\leq 2^{l+3}$
\[ 
\left| \partial^\alpha_\xi \left(m_l(\xi) \right) \right| 
\leq\left| \partial^\alpha_\xi \left(\frac{2^{-ls}\langle \xi \rangle ^s\eta(2^{-l-2}\xi)}{\langle{\Lambda^{-1}\xi}\rangle^s} \right) \right| 
\leq C_{\alpha}2^{-ls}\langle\xi\rangle^{s-|\alpha|} \leq C_{\alpha} \langle\xi\rangle^{-|\alpha|} \] 
As this bound is independent of $l$, $\vec{M}$ is a vector valued Calder\'on-Zygmund operator \cite{GCF85} and, thus, is bounded from $L^r(\ell^2,u)$ to $L^r(\ell^2,u)$ for any $u\in A_r$ which proves 
\begin{align*}  \hspace{.75cm} \left\| \left\{ {Q}_l \derivslx{s}{\Lambda} \left(\sum_{j=3}^\kappa Q_jf P_j g\right) \right\}_{-\infty}^{\infty} \right\|_{L^r(\ell^2,u)} \hspace{-.5cm} & \lesssim\left\| \left\{ 2^{ls}  {Q}_l \left(\sum_{j=3}^\kappa Q_jf P_j g\right) \right\}_{-\infty}^{\infty} \right\|_{L^r(\ell^2,u)}. 
\end{align*} 
Next, noting that $Q_l ( Q_jf P_j g)=0$ if $l>j+4$, using  Minkowski's integral inequality,


\[   \left\| \left\{ 2^{ls}  \tilde{Q}_l \left(\sum_{j=3}^\kappa Q_jf P_j g\right) \right\}_{ l=-\infty}^\infty \right\|_{L^r(\ell^2,u)}  \leq C  \left\| \sum_{t=-4}^\infty \left(\sum _{j=3}^{\kappa}\left|2^{(j-t)s}\tilde{Q}_{j-t}\left(  Q_{j}f P_{j}g\right)\right|^2 \right)^{1/2} \right\|_{L^r(u)}.\] 

As the sum $\sum_{t=-4}^\infty 2^{-ts}$ is a finite constant depending only on $s$, it suffices to estimate 
\[  \left\|  \left\{ 2^{js}\tilde{Q}_{j-t}\left(  Q_{j}f P_{j}g\right)\right\} _{j=3}^{\kappa} \right\|_{L^r(\ell^2,u)} \] 
uniformly in $t.$ Denote by $\mathcal{M}$ the Hardy-Littlewood Maximal function. Here $|Q_{l}g(x)|=|P_lg(x)-P_{l-1}g(x)|\leq C \mathcal{M}|g|$, and thus 

\begin{align*}
    \left\|  \left\{ 2^{js}\tilde{Q}_{j-t}\left(  Q_{j}f P_{j}g\right)\right\} _{j=3}^{\kappa} \right\|_{L^r(\ell^2,u)}\lesssim & \left\|  \left\{ \mathcal{M}\left(  2^{js}Q_{j}f P_{j}g\right)\right\} _{j=3}^{\kappa} \right\|_{L^r(\ell^2,u)}\\
    \lesssim & \left\|  \left\{   2^{js}Q_{j}f P_{j}g\right\} _{j=3}^{\kappa} \right\|_{L^r(\ell^2,u)}.
\end{align*} 

This last estimate is the observation Fefferman and Stein's (\cite{FSmax}) vector-valued maximal function (with $q=2$) is well known to be bounded on $L^r(\ell^2,u)$ for $u\in A_r$ (see \cite{AJ80}).

Similarly, 
\[|2^{js}Q_jf P_{j-3}g |\lesssim 2^{js}|Q_jf | \mathcal{M}g.  \]
Whence, 
\[\left\|  \left\{ \left(  2^{js}Q_{j}f P_{j}g\right)\right\} _{j=3}^{\kappa} \right\|_{L^r(\ell^2,u)} \leq C 
\left\|  \mathcal{M}(g)\left\{ \left(  2^{js}Q_{j}f\right)\right\} _{j=3}^{\kappa} \right\|_{L^r(\ell^2,u)} \] 
Therefore, using that $u=u_1v_1$ and  H\"older's inequality, 
\[ \left\|  \mathcal{M}(g)\left\{ \left(  2^{js}Q_{j}f\right)\right\} _{j=3}^{\kappa} \right\|_{L^r(\ell^2,u)} \leq  C \left\| \mathcal{M}g\right\|_{L^{q_1}(v_1^{q_1/r})}\left\|\left\{ \left(  2^{js}Q_{j}f\right)\right\} _{j=3}^{\kappa} \right\|_{L^{p_1}(\ell^2,u_1^{p_1/r})}\] 

As $v_1^{q_1/r}\in A_{q_1}$, 
\[ \left\| \mathcal{M}g \right\|_{L^{q_1}(v_1^{q_1/r})}\leq C  \left\|g \right\|_{L^{q_1}(v_1^{q_1/r})} \]
and it remains to show that, for a constant independent of $\Lambda,$
\[ \left\| \{ 2^{js}Q_jf \}_{j=3}^\kappa \right\|_{L^{p_1}(\ell^2,u_1^{p_1/r})} \lesssim \left\|\derivslx{s}{\Lambda}f  \right\|_{L^{p_1}(u_1^{p_1/r})}.\] 

To do so, recall that $Q_j$ is supported on $\{\xi:2^{j-1}|\xi|\leq 2^{j+1}\}$. Let $\zeta(\xi)$ a Schwartz cut-off function such that $\zeta(\xi)=1$ for $1/2\leq|\xi|<2$ and $\zeta(\xi)=0$ for $|\xi|<1/4$ and $|\xi|>4$.  Then, 
\[ \widehat{ 2^{js}Q_jf} = (\eta(2^{-j}\xi)-\eta(2^{-j+1}\xi))\zeta(2^{-j}\xi)\left(\frac{2^{js}\langle{\Lambda^{-1}\xi}\rangle^s}{\langle \xi \rangle ^s}\right) \frac{\langle \xi \rangle ^s}{\langle{\Lambda^{-1}\xi}\rangle^s}\hat{f}(\xi) .  \]
Using \eqref{eq:inverseDecay} and support considerations, 
\[ \left| \partial^\alpha_\xi  \frac{2^{js}\langle{\Lambda^{-1}\xi}\rangle^s}{\langle \xi \rangle ^s}  \right| 
\leq C_{\alpha}2^{js}\langle\xi\rangle^{-s-|\alpha|}\leq C_{\alpha} \langle\xi\rangle^{-|\alpha|}.\] 

Applying weighted Littlewood-Paley theory yet again,
\[  \left\| \{ Q_j (\derivslx{s}{\Lambda}f )\}_{j=3}^\kappa \right\|_{L^{p_1}(\ell^2,u_1^{p_1/r})} \leq C \left\|\derivslx{s}{\Lambda}f  \right\|_{L^{p_1}(u_1^{p_1/r})}.\] 
Finally, 
\[   \|\derivslx{s}{\Lambda}( \sum_{j=1}^\kappa Q_jf\cdot P_{j} g )\|_{L^r(u)} \leq C \left\|\derivslx{s}{\Lambda}f  \right\|_{L^{p_1}(u_1^{p_1/r})} \left\|g\right\|_{L^{q_1}(v_1^{q_1/r})} .\] 
The same argument with the roles of $f$ and $g$ reversed shows that 
\[  \|\derivslx{s}{\Lambda}( \sum_{j=3}^\kappa Q_jg\cdot P_{j-3} f )\|_{L^r(u)} \leq C\left\|f\right\|_{L^{p_2}(u_2^{p_2/r})} \left\|\derivslx{s}{\Lambda}g  \right\|_{L^{q_2}(v_2^{q_2/r})}  .\]

which completes the arguement.

\end{proof}
\bibliographystyle{abbrv}	
\bibliography{ELbib}

\end{document}